\documentclass[10pt,leqno]{article}
\usepackage{verbatim}
\usepackage{amsthm}
\usepackage{amsfonts}
\usepackage{amsmath}
\usepackage{amssymb}
\usepackage{verbatim, color}

\begin{document}

\thispagestyle{empty}

\newtheorem{Lemma}{\bf LEMMA}[section]
\newtheorem{Theorem}[Lemma]{\bf THEOREM}
\newtheorem{Claim}[Lemma]{\bf CLAIM}
\newtheorem{Corollary}[Lemma]{\bf COROLLARY}
\newtheorem{Observation}[Lemma]{\bf OBSERVATION}
\newtheorem{Proposition}[Lemma]{\bf PROPOSITION}
\newtheorem{Example}[Lemma]{\bf EXAMPLE}
\newtheorem{Fact}[Lemma]{\bf FACT}
\newtheorem{Definition}[Lemma]{\bf DEFINITION}
\newtheorem{Notation}[Lemma]{\bf NOTATION}

\newcommand{\restrict}{\mbox{$\mid\hspace{-1.1mm}\grave{}$}}
\newcommand{\covers}{\mbox{$>\hspace{-2.0mm}-{}$}}
\newcommand{\covered}{\mbox{$-\hspace{-2.0mm}<{}$}}
\newcommand{\notcover}{\mbox{$>\hspace{-2.0mm}\not -{}$}}

\newcommand{\boldalpha}{\mbox{\boldmath $\alpha$}}
\newcommand{\boldbeta}{\mbox{\boldmath $\beta$}}
\newcommand{\boldgamma}{\mbox{\boldmath $\gamma$}}
\newcommand{\boldxi}{\mbox{\boldmath $\xi$}}
\newcommand{\boldlambda}{\mbox{\boldmath $\lambda$}}
\newcommand{\boldmu}{\mbox{\boldmath $\mu$}}

\newcommand{\barzero}{\bar{0}}

\newcommand{\sfq}{{\sf q}}
\newcommand{\sfe}{{\sf e}}
\newcommand{\sfk}{{\sf k}}
\newcommand{\sfr}{{\sf r}}
\newcommand{\sfc}{{\sf c}}
\newcommand{\restr}{\negmedspace\upharpoonright\negmedspace}

\title{Varieties of Regular Pseudocomplemented de Morgan Algebras}
\author{M.~E.~Adams, H.~P.~Sankappanavar, and J\'{u}lia Vaz de Carvalho\footnote{This work was partially supported by the Funda\c{c}\~{a}o para a Ci\^{e}ncia e Tecnologia (Portuguese Foundation for Science and Technology) through the projects UID/MAT/00297/2013 and UID/MAT/00297/2019 (Centro de Matem\'{a}tica e Aplica\c{c}\~{o}es).} }
\date{}
\maketitle
{\color{red} This paper has appeared in the journal Order (published online on January 10, 2020).}

\begin{abstract}

In this paper, we investigate the varieties  $\mathbf{M}_n$ and $\mathbf{K}_n$  of regular pseudocomplemented de Morgan and Kleene algebras of range $n$, respectively. 
 Priestley duality as it applies to pseudocomplemented de Morgan algebras is used. We characterise the dual spaces of the simple (equivalently, subdirectly irreducible) algebras in $\mathbf{M}_n$ and explicitly describe the dual spaces of the simple algebras in $\mathbf{M}_1$ and $\mathbf{K}_1$. We show that the variety $\mathbf{M}_1$ is locally finite, but this property does not extend to $\mathbf{M}_n$ or even $\mathbf{K}_n$ for $n\geq 2$. We also show that the lattice of subvarieties of $\mathbf{K}_1$ is an $\omega+1$ chain and the cardinality of the lattice of subvarieties of either $\mathbf{K}_2$ or $\mathbf{M}_1$ is $2^{\omega}$. A description of the lattice of subvarieties of $\mathbf{M}_1$ is given.

\end{abstract}

\medskip


\maketitle

\thispagestyle{empty}

\section{Introduction} \label{SA}

A {\it pseudocomplemented de Morgan algebra} ($pm$-algebra for short) is an algebra $(L;\wedge ,\vee , ^\ast, ^\prime ,0,1)$ of signature $(2,2,1,1,0,0)$ such that $(L;\wedge ,\vee , ^\ast, 0,1)$ is a pseudocomplemented distributive algebra and $(L;\wedge , \vee , ^\prime ,0,1)$ is a de Morgan algebra, that is $(L;\wedge ,\vee ,0,1)$ is a bounded distributive lattice, 
$^\ast$  satisfies, for $x,y\in L$, $x\wedge y=0$ if and only if $y\leq x^{\ast}$, 
whilst $^\prime$ satisfies
$(x\vee y)^\prime = x^\prime \wedge y^\prime $, $(x\wedge y)^\prime = x^\prime \vee y^\prime $, $0^\prime = 1$, $1^\prime  = 0$, and $x^{\prime \prime}  = x$.\\

Since, as shown by Ribenboim \cite{Ri49}, the class of  pseudocomplemented distributive algebras form a variety, so too does the class of pseudocomplemented de Morgan algebras.\\

Pseudocomplemented de Morgan algebras were first investigated in Romanowska \cite{Ro81} and have been considered by several authors since (see, for example, Denecke \cite{De87}, Gait\'{a}n \cite{Ga98}, Guzm\'{a}n  and Squier \cite{GaSq94}, as well as \cite{Sa86}, \cite{Sa87a}, \cite{Sa87b}, \cite{SaVa14}, \cite{Wang17}).  Romanowska's characterisation of finite subdirectly irreducible $pm$-algebras in  \cite{Ro81}, was extended to a characterisation of all subdirectly irreducible $pm$-algebras in  \cite{Sa86} (for those algebras which are {\it non-regular}) and in \cite{Sa87b} (for those algebras which are  {\it regular}). An algebra is {\it regular} if any two congruences on it are equal whenever they have a class in common.  It was in \cite{Sa86} that 
regularity in  $pm$-algebras was first considered.   It is better understood in the context of regular double $p$-algebras where they have been studied extensively (see, for example, Beazer \cite{Be76}, Katri\v{n}\'{a}k \cite{Ka73}, Koubek and Sichler \cite{KoSi90}, Taylor \cite{Ta16}, Varlet \cite{Va72}, as well as \cite{AdSaVc19}).\\

As observed by Romanowska, for any element $a$ in a $pm$-algebra, $a^+ = a^{\prime \ast\prime }$ is its dual pseudocomplement.  That is, every $pm$-algebra $L=(L;\wedge ,\vee ,^\ast, ^\prime ,0,1)$ gives rise to a  distributive double $p$-algebra $L_{dp}=(L;\wedge ,\vee ,^\ast, ^+,0,1)$. The {\it determination} congruence $\Phi$ on a double $p$-algebra $(L;\wedge ,\vee ,^\ast, ^+,0,1)$ is  the congruence given by, for $x,y\in L$
\[  x\equiv y\ \! (\Phi ) \mbox{ iff } x^\ast = y^\ast \mbox{ and } x^+ =y^+. \]
As shown by Varlet \cite{Va72}, a double $p$-algebra is regular if and only if $\Phi = \triangle$,  the equality congruence. The {\it Moisil} congruence $\Psi$ on a  $pm$-algebra  $(L;\wedge ,\vee ,^\ast, ^\prime ,0,1)$ is the congruence given by, for $x,y\in L$,
\[  x\equiv y\ \! (\Psi ) \mbox{ iff } x^\ast = y^\ast \mbox{ and } x^{ \prime \ast} =y^{ \prime \ast} \]
 As shown in  \cite{Sa86},  a  $pm$-algebra $L$ is regular if and only if  $\Psi = \triangle$. Noticing that $x^{ \prime \ast} =y^{ \prime \ast}$ is equivalent to $x^{ \prime \ast \prime} =y^{ \prime \ast \prime}$, we have that $L$ is regular if and only if the distributive double $p$-algebra $L_{dp}$ is regular.
That the class of regular double $p$-algebras is a variety, namely the subvariety of distributive double $p$-algebras determined by the inequality
\[  x\wedge x^+ \leq y\vee y^\ast, \]
was shown by Katri\v{n}\'{a}k \cite{Ka73}.
It follows immediately that the class of regular $pm$-algebras is a variety, namely the subvariety of  $pm$-algebras determined by the inequality
\[  x\wedge x^{ \prime \ast \prime} \leq y\vee y^\ast. \]

Let $(L;\wedge ,\vee ,^\ast, ^\prime ,0,1)$ be a   $pm$-algebra.
For $x\in L$ and $ n<\omega$, define $x^{n(\prime \ast)}$ recursively as follows:  \[x^{0(\prime \ast)} = x \mbox { and } x^{(n+1)(\prime \ast)} = (x^{n(\prime \ast)})^{\prime \ast}.\] \par
 For any element $x$ of a  distributive double $p$-algebra, $x\geq x^{+\ast}$, that is, for any $x\in L$, $x\geq x^{\prime \ast\prime \ast}$.  Further, since for  $x,y\in L$, $(x\vee y)^\ast = x^\ast\wedge y^\ast$, it follows that
$x\wedge x^{\prime \ast}\geq x^{\prime \ast}\wedge x^{\prime \ast\prime \ast} = (x^{\prime }\vee x^{\prime \ast\prime })^\ast = (x\wedge x^{\prime *})^{\prime \ast}$.  In other words, it is always the case that
\[x\wedge x^{\prime \ast}\geq (x\wedge x^{\prime \ast})^{\prime \ast}\geq (x\wedge x^{\prime \ast})^{2(\prime \ast)}\geq \ldots\] As such, $L$ is said to have {\it finite range} providing, for each $x\in L$, there exists $n<\omega$ such that $(x\wedge x^{\prime *})^{n(\prime \ast)} = (x\wedge x^{\prime \ast})^{(n+1)(\prime \ast)}$.
Further, if, for some
$ n<\omega$, $(x\wedge x^{\prime \ast})^{n(\prime \ast)} = (x\wedge x^{\prime \ast})^{(n+1)(\prime \ast)}$ for every $x\in L$, then $L$ is said to {\it be of range} $n$. Obviously, if $L$ is of range $n$, it is also of range $m$, for any $m\geq n$.

A variety of (regular) $pm$-algebras is said to be of \emph{range n} if all its members are of range $n$. For a variety $\mathbf V$, $L_V (\mathbf V)$ denotes the lattice of all subvarieties of $\mathbf V$.

Although it appears to be somewhat ad hoc, the notion of finite range occurs naturally in the study of $pm$-algebras.  For example, in \cite{Sa86} it is shown that for $L$ to be regular with finite range is equivalent to the compact elements  forming a Boolean sublattice of the congruence lattice of $L$. Also from \cite{Sa86}, for each $n<\omega$, the variety of regular $pm$-algebras of range $n$ is a discriminator variety and so, in particular, it has equationally definable principal congruences, every subdirectly irreducible algebra is simple, and any subalgebra of a simple algebra is simple. 
In  \cite{Sa87a}, it is shown that $L$ has finite range if and only if every principal congruence on $L$ is a join of finitely many principal lattice congruences on $L$.\\

 Further, in \cite{Sa87a}, the variety $\mathbf{PM}_0$ (there denoted by $V_0$) of the $pm$-algebras of range $0$ is considered and its subdirectly irreducible algebras are described in order to determine the lattice $L_V(\mathbf{PM}_0)$ of its subvarieties. We remark that $L_V(\mathbf{PM}_0)$ is isomorphic to the 11-element lattice ${\mathbf 1}\oplus ({\mathbf 2}\times{\mathbf 5})$ (a correction to the lattice presented in \cite[Corollary 6.7]{Sa87a}, since the subvarieties there denoted by $V(L_3)$ and $V(D_4)$ are actually incomparable). Having determined $L_V(\mathbf{PM}_0)$, it was natural to ask for a description of  $L_V(\mathbf{PM}_1)$, $L_V(\mathbf{PM}_2),\ldots$ where $\mathbf{PM}_n$ denotes the variety of $pm$-algebras of range $n$. That question, raised in \cite{Sa87a}, was 
the primary motivation behind this paper.\\

Let $\mathbf{M}$ denote the variety of regular pseudocomplemented de Morgan algebras and, for $ n<\omega$, let $\mathbf{M}_n\subseteq \mathbf{PM}_n$ denote the subvariety of $\mathbf{M}$ that satisfies $(x\wedge x^{\prime \ast})^{n(\prime \ast)} \approx (x\wedge x^{\prime \ast})^{(n+1)(\prime \ast)}$, that is, $\mathbf{M}_n=\mathbf{M}\cap \mathbf{PM}_n$ is the variety of regular pseudocomplemented de Morgan algebras of range $n$.
In this paper, it will be  shown that, even for the better behaved subvariety $\mathbf{M}_n$ of $\mathbf{PM}_n$, $L_V(\mathbf{M}_n)$ is already quite complex.\\

As shown by Kalman \cite{Ka58}, the lattice of subvarieties  of the variety of de Morgan algebras  is a $4$-element chain whose members are, from the least to the greatest, the trivial variety, the variety of Boolean algebras, the variety of {\it Kleene algebras}, and the variety of de Morgan algebras.  As determined by Kalman,
a {\it Kleene algebra} is a de Morgan algebra that satisfies the inequality
\[ x\wedge x^\prime \leq y\vee y^\prime .    \]
In the present context, let $\mathbf{K}$ denote the variety of regular pseudocomplemented Kleene algebras and, for $ n<\omega$,  let $\mathbf{K}_n$ denote the subvariety of $\mathbf{K}$ that satisfies $(x\wedge x^{\prime \ast})^{n(\prime \ast)} \approx (x\wedge x^{\prime \ast})^{(n+1)(\prime \ast)}$, that is, $\mathbf{K}_n=\mathbf{K}\cap \mathbf{PM}_n$ is the variety of regular pseudocomplemented Kleene algebras of range $n$. Obviously, $\mathbf{K}_n\subseteq \mathbf{M}_n\subseteq \mathbf{PM}_n$.
Typically there is a significant difference between Kleene algebras and de Morgan algebras and, although this is reflected initially in a sharp difference between $L_V(\mathbf{K}_0)$ and $L_V(\mathbf{M}_0)$, as well as between $L_V(\mathbf{K}_1)$ and $L_V(\mathbf{M}_1)$, thereafter this difference is short lived.\\

 A simple inspection of the description of the subdirectly irreducible algebras in the variety $\mathbf{PM}_0$ presented in \cite[Theorem 6.5]{Sa87a} leads to the conclusion that  $L_V(\mathbf{M}_0)$ is isomorphic to the $5$-element lattice ${\mathbf 1}\oplus ({\mathbf 2}\times{\mathbf 2})$ and that
$L_V({\mathbf{K}_0})$ is isomorphic to a $3$-element chain. In Sections \ref{SE} and \ref{SF} we will show that $L_V({\mathbf{K}_1})$ is isomorphic to an $\omega +1$ chain (Theorem \ref{E2}), $|L_V({\mathbf{M}_1})| = 2^\omega$ (Corollary \ref{E4}) and
$|L_V({\mathbf{K}_2})| = 2^\omega$ (Theorem \ref{E8}). Summarising we have the following theorems.

\begin{Theorem} \label{A1}
For the variety ${\mathbf{K}}$, we have

{\rm (1)} $L_V({\mathbf{K}_0})$ is isomorphic to a $3$-element chain,

{\rm (2)}
$L_V({\mathbf{K}_1})$ is isomorphic to an $\omega +1$ chain, 

{\rm (3)}
$|L_V({\mathbf{K}_2})| = 2^\omega$.\hfill{$\Box$}
\end{Theorem}

\begin{Theorem} \label{A2}
For the variety ${\mathbf{M}}$, we have

{\rm (1)} $L_V({\mathbf{M}_0})$ is isomorphic to the $5$-element lattice ${\mathbf 1}\oplus ({\mathbf 2}\times{\mathbf 2})$,

{\rm (2)} $|L_V({\mathbf{M}_1})| = 2^\omega$.\hfill{$\Box$}
\end{Theorem}
   
Local finiteness in the varieties $\mathbf{M}_n$ and $\mathbf{K}_n$, $1\leq n<\omega$, is also analysed in Theorems
 \ref{D3} and \ref{D6} allowing us to establish the following 
result.

\begin{Theorem} \label{A3} 
For the varieties ${\mathbf{K}}$ and ${\mathbf{M}}$,

{\rm (1)}  $\mathbf{M}_1$ is locally finite {\rm (}de facto, so too is $\mathbf{K}_1${\rm )}, 

{\rm (2)} ${\mathbf{K}_2}$ is not locally finite, in fact the free algebra on one generator  $F_{\mathbf{K}_2}(1)$ is infinite
 {\rm (}de facto, for $n\geq 2$,  ${\mathbf{K}_n}$ and ${\mathbf{M}_n}$  are not locally finite and the respective free algebras on one generator are infinite{\rm )}. \hfill{$\Box$}
\end{Theorem}

Even though $L_V({\mathbf{M}_1})$ is uncountable, local finiteness of $\mathbf{M}_1$ allows us to have a genuine understanding of the structure of $L_V({\mathbf{M}_1})$, the key to which is given in Theorem \ref{E3}. The discussion following Corollary \ref{E4} contains the details.  At the cost of some transparency, that discussion is summarised in Theorem \ref{E6}.\\

To establish the aforementioned results,  Priestley duality as it applies to pseudocomplemented de Morgan algebras is used extensively and a brief outline of this duality is given in Section \ref{SB}. In Section \ref{SC} we characterise, via their dual spaces, the regular $pm$-algebras of range $n$. 
It will be seen that the apparently ad hoc nature of finite range is not ad hoc at all and, in fact, is a natural property of Priestley duality, namely the notion of $\zeta$-distance, as seen in Section \ref{SC}.
A characterisation, via dual spaces, of the simple (equivalently, subdirectly irreducible) algebras in $\mathbf{M}_n$ and a description of the dual spaces of the simple algebras in $\mathbf{M}_1$ and $\mathbf{K}_1$ are given in Section \ref{SD}. It is also in this section that we prove Theorem \ref{A3}. The contents of Sections \ref{SE} and \ref{SF} have already been described above. Finally, in Section \ref{SG} we present some remarks on connections between regular $pm$-algebras and de Morgan Heyting algebras.\\

\section{Preliminaries}\label{SB}

We refer to \cite{BuSa81} and to \cite{BaDw74} for basic concepts and results on universal algebra and on distributive $p$-algebras, respectively. We point out that, for us, subdirectly irreducible algebras and simple algebras are non-trivial.\\

We will use Priestley duality as it applies to pseudocomplemented de Morgan algebras. Here we give a brief outline.\\

Let $(P;\leq)$ be a partially ordered set. We denote by $\mathrm{Min}(P)$ and $\mathrm{Max}(P)$  the sets of minimal elements of $P$ and of maximal elements of $P$, respectively. 
For any $X\subseteq P$, let $(X]=\{y\in P\colon y\leq x \mbox{ for some } x\in X\}$, and $[X)=\{y\in P\colon y\geq x \mbox{ for some } x\in X\}$. The subset $X$ is said to be {\it decreasing} if $(X]=X$, and {\it increasing} if $[X)=X$. Define also $\mathrm{Min}(X)=\mathrm{Min}(P)\cap (X]$ and $\mathrm{Max}(X)=\mathrm{Max}(P)\cap [X)$. Should $X=\{x\}$, the sets $(X]$, $[X)$, $\mathrm{Min}(X) $ and $\mathrm{Max}(X)$ will be denoted by $(x]$, $[x)$, $\mathrm{Min}(x) $ and $\mathrm{Max}(x)$, respectively. \\

An ordered topological space  $P=(P;\tau ,\leq )$ is a {\it Priestley space} if it is a compact totally order-disconnected space, that is $(P;\tau )$ is a compact topological space, $(P;\leq )$ is a partially ordered set, and, for $x,y\in P$, if $y\not\leq x$, then there exists a clopen decreasing set $X$ such that $x\in X$ and $y\not\in X$.  In \cite{Pr70}, Priestley showed that the category of all Priestley spaces together with all continuous order-preserving maps (that is, for a continuous map $\varphi \colon P\to Q$, if $x\leq y$, then $\varphi (x)\leq \varphi (y)$) is dually equivalent to the category of all bounded distributive lattices together with all  $\{0,1\}$-lattice homomorphisms (see also \cite{Pr84}).\\

Priestley \cite{Pr75} extended her duality to the category of distributive pseudocomplemented lattices by showing that the category of distributive pseudocomplemented lattices together with all lattice homomorphisms that preserve $^\ast$ is dually equivalent to the category of all {\it $p$-spaces} together with all {\it $p$-morphisms}.  A {\it $p$-space} is a Priestley space $P=(P;\tau ,\leq )$ such that, whenever $X\subseteq P$ is a clopen decreasing set, then $[X)$ is clopen. Given $p$-spaces $P$ and $Q$, a map $\varphi \colon P\to Q$ is a {\it $p$-morphism} if it is continuous order-preserving and, for any $x\in P$, ${\rm Min} (\varphi (x))\subseteq \varphi ({\rm Min}(x))$, which implies that ${\rm Min} (\varphi (x))= \varphi ({\rm Min}(x))$.\\

Meanwhile, Cornish and Fowler \cite{CoFo77} extended Priestley's duality to the category of de Morgan algebras by showing that the category of de Morgan algebras together with all lattice homomorphisms that preserve $^\prime$ is dually equivalent to the category of all {\it $m$-spaces} together with all {\it $m$-morphisms}.  An {\it $m$-space} is a quadruple $P=(P;\tau ,\leq, \zeta )$ such that $(P;\tau,\leq)$ is a Priestley space and $\zeta$ is a continuous order-reversing map from $P$ to $P$ satisfying $\zeta^2(x) = x$ for every $x\in P$. Given $m$-spaces $P$ and $Q$, a map $\varphi \colon P\to Q$ is an {\it $m$-morphism} if it is continuous order-preserving and $\varphi \circ\zeta = \zeta\circ \varphi $. \\

To summarise: $P=(P;\tau, \leq,\zeta)$ is a {\it $pm$-space} if $(P;\tau, \leq)$ is a Priestley space, $[X)$ is clopen for every clopen decreasing set $X\subseteq P$, $\zeta$ is a continuous order-reversing involution on $P$; for $pm$-spaces $P$ and $Q$, $\varphi \colon P\to Q$ is a {\it $pm$-morphism} if $\varphi$ is continuous, order-preserving, ${\rm Min} (\varphi (x))\subseteq \varphi ({\rm Min}(x))$ for every $x\in P$, and $\varphi \circ\zeta = \zeta\circ \varphi $.  The category of all pseudocomplemented de Morgan algebras together with all homomorphisms is dually equivalent to the category of all  $pm$-spaces together with all  $pm$-morphisms. In particular, for a $pm$-algebra $(L;\wedge ,\vee ,^\ast, ^\prime ,0,1)$, its dual space is the $pm$-space $D(L)=(P;\tau ,\leq, \zeta )$ where $P$ is the set of prime ideals of $L$ ordered by inclusion, the topology $\tau$ has sub-basis $\{ X_a\colon a\in L\}\cup \{ P\setminus X_a\colon a\in L\}$ where $X_a = \{ I\in P\colon a\not\in I\}$ and $\zeta$ is defined by $\zeta (I) = \{ a\in L\colon a^\prime\not\in I\}$ for every $I\in P$.
For a $pm$-space $(P;\tau ,\leq, \zeta )$, its dual $pm$-algebra is the algebra $E(P)=(L;\cap ,\cup ,^\ast, ^\prime ,\emptyset ,P)$ where $L$ consists of all clopen decreasing subsets of $P$ and, for a clopen decreasing $X\subseteq P$, $X^\ast = P\setminus [X)$ and $X^\prime = P\setminus \zeta^{-1}(X) = P\setminus \zeta (X) = \zeta (P\setminus X)$.\\

For bounded distributive lattices $L$ and $K$ whose dual  Priestley spaces are $P$ and $Q$, respectively, each homomorphism $f\colon L\to K$ is associated with a continuous order-preserving map $\varphi\colon Q\to P$ which is defined by $\varphi(J)=f^{-1}(J)$, for any prime ideal $J$ of $K$.  Under Priestley duality, $f$ is one-to-one if and only if $\varphi$ is onto and $f$ is onto if and only if $\varphi$ is an order-embedding.  From the latter, it follows that homomorphic images of $L$ correspond to closed subspaces of $P$ (ordered by the restriction of the order in $P$).  For $pm$-spaces, we know that the $pm$-morphisms are 
continuous order-preserving maps $\varphi \colon Q\to P$  such that, for every $x\in Q$,  ${\rm Min} (\varphi (x))\subseteq \varphi ({\rm Min}(x))$  and $\varphi \circ\zeta = \zeta\circ \varphi $.  Hence, it follows that the closed subspaces of $P$, for which $\mathrm{Min}(x)$ is a subset whenever $x$ is a member and are also closed under $\zeta$, correspond to the homomorphic images of $L$.  We refer to all such subspaces as {\it $pm$-subspaces} of $P$.\\

As observed by Cornish and Fowler \cite{CoFo79}, an $m$-space $(P;\tau ,\leq, \zeta )$ corresponds to a Kleene algebra precisely when either $x\leq \zeta (x)$ or $\zeta (x)\leq x$  for every $x\in P$.
In particular, the category of all pseudocomplemented Kleene  algebras together with all homomorphisms is dually equivalent to the category of all $pm$-spaces  $(P;\tau ,\leq, \zeta )$ such that  either  $x\leq \zeta (x)$ or $\zeta (x)\leq x$ for every $x\in P$, together with all $pm$-morphisms.\\

If a $pm$-morphism is a homeomorphism and an order-isomorphism, we say that it is a {\it $pm$-isomorphism} and that the $pm$-spaces involved are {\it $pm$-isomorphic}.\\
 
Notice that if $(P;\tau ,\leq, \zeta )$ is a $pm$-space then $\zeta$ is a homeomorphism and a dual order-isomorphism from $P$ onto $P$.\\

It is well known that finite subsets of a Hausdorff space are closed and, consequently, an element $x$ is an isolated point if and only if $\{x\}$ is clopen. Recall also that a compact Hausdorff space is finite if and only if the topology is the discrete one. 
{\it Stone spaces} (that is, compact totally disconnected  topological spaces) and Priestley spaces are, in particular, compact and Hausdorff.\\

Given a closed subset $X$ of a Priestley space, $[X)$ and $(X]$ are closed and if $X$ is decreasing and $y\notin X$ there exists a clopen increasing set $Y$ such that $y\in Y$ and $Y\cap X=\emptyset$.\\

If $(P;\tau,\leq,\zeta)$ is the $pm$-space of a pseudocomplemented de Morgan algebra $L$, then $(P;\tau,\leq)$ is the Priestley space of the distributive double $p$-algebra $L_{dp}$. Consequently, ${\rm Min}(P)$ and ${\rm Max} (P)$ are closed, since the set of minimal elements in the Priestley space of any distributive pseudocomplemented lattice is closed \cite{Pr75} and dually the set of maximal elements in the Priestley space of any double $p$-algebra is also closed. \\

Combining an observation of Cornish and Fowler \cite{CoFo79} with a result due to Adams (cf. \cite{Pr75}), we know that given a pseudocomplemented de Morgan algebra $L$ whose $pm$-space is $P=(P;\tau,\leq,\zeta)$, there is a lattice-isomorphism from the lattice of congruences of $L$ onto the lattice of open sets $X$ of $P$ that satisfy $[{\rm Min}(P)\cap X)\subseteq X$ and $\zeta(X)\subseteq X$ (thus $\zeta(X)=X$). These open sets also satisfy $({\rm Max}(P)\cap X]\subseteq X$, since $({\rm Max}(P)\cap X]=\zeta([{\rm Min}(P)\cap X))\subseteq \zeta (X)=X$ (that $({\rm Max}(P)\cap X]\subseteq X$ also follows from the fact that any congruence on $L$ is also a congruence on $L_{dp}$).

\section{Regularity and finite range}\label{SC}

In this section, we characterise, via their dual spaces, the regular pseudocomplemented de Morgan algebras and which amongst  these algebras have finite range $n$.
Although at first flush the notion of finite range may have seemed a little artificial, as Corollary \ref{C4B}  shows, in the context of pseudocomplemented de Morgan algebras, it really does arise naturally.\\

\begin{Theorem}\label{C1} Let $L=(L;\wedge ,\vee ,^\ast, ^\prime ,0,1)$ be a $pm$-algebra and  $(P;\tau, \leq, \zeta)$ its dual  $pm$-space. We have that 
$L$ is regular if and only if $(P;\leq )$ has height at most $1$, that is, all chains in $(P;\leq )$ have at most $2$ elements.
\end{Theorem}
\begin{proof}
 Recall that $L$ is regular if and only if its associated distributive double $p$-algebra $L_{dp}$ is regular, which, by \cite[Theorem 3]{Va68}, is equivalent to every chain of prime ideals of $L$ having at most two elements, that is $(P;\leq )$ having height at most $1$. 
\end{proof}

In the context of the previous result, we have that $L$ is regular if and only if $P={\rm Min}(P)\cup {\rm Max}(P)$.\\

In what follows, we consider that $n<\infty$, for all $ n<\omega$.\\

Let $(P;\leq )$ be a poset.  The comparability graph of $P$ is the undirected graph with $P$ as the set of vertices and in which there is an edge between $x$ and $y$ if and only if $x \neq y$ and $x$ and $y$ are comparable in the poset. For $x,y \in P$, the {\it distance between} $x$ {\it and} $y$, denoted $\ell(x,y)$, is $0$ if $x=y$ and otherwise is the length of a shortest (finite) path between $x$ and $y$ in the comparability graph of $P$, if such a path exists. In this case  we say that the distance between $x$ and $y$ is {\it finite},  otherwise the distance is said to be {\it infinite} and we write $\ell(x,y)=\infty$.\\

An {\it order component} $Q\subseteq P$ is a non-empty set such that, for any $x,y\in Q$, $\ell(x,y)$ is finite and, for any $x\in Q$ and any $y\not\in Q$, $\ell(x,y)$ is infinite. The order components of $(P;\leq)$ form a partition of $P$ and we denote by $Q_x$ the order component to which $x\in P$ belongs. \\
For $ X\subseteq P$ and $x\in P$, the {\it distance of $x$ from $X$}, denoted $\ell(x,X)$, is the least   $\ell(x,y)$ with $y\in X$ if $X\neq\emptyset$, and $\ell(x,\emptyset)=\infty$. \\

Let $P$ be a $pm$-space. Since $\zeta$ is an involutive dual order-isomorpism on $P$, it is obvious that, for any $x,y\in P$ and $X\subseteq P$, we have  $\ell(x,y)=\ell (\zeta(x),\zeta(y))$, $\ell(x,\zeta(y))=\ell (\zeta(x),y)$ and $\ell(x, X)=\ell(\zeta(x),\zeta (X))$. Moreover, if $Q$ is an order component of $P$, then $\zeta(Q)$ is also an order component and, consequently, either $Q\cap \zeta(Q)=\emptyset$ or $Q=\zeta (Q)$.

\begin{Lemma}\label{C2} Let $P$ be a $pm$-space of height at most $1$ and  $X\subseteq P$ be a clopen decreasing set. Then we have:\\
 For $n<\omega$,
 \[ X^{n(\prime \ast)}  = \{ x\colon \ell(x, (P\setminus X))>n  \}, \mbox{ if } n \mbox{ is even,} \]
and,
\[ X^{n(\prime \ast)}  = \{ x\colon \ell (x,\zeta (P\setminus X))>n  \}, \mbox{ if } n \mbox{ is odd}. \]  
 
Moreover,  $\ell(x,(P\setminus X))$ is infinite precisely when
$Q_x\subseteq X$ and $\ell(x,\zeta (P\setminus X))$ is infinite precisely when
$\zeta (Q_x)\subseteq X$ and
 
\end{Lemma}
\begin{proof} Let $X\subseteq P$ be a clopen decreasing set.\\

 We begin by proving the last statement. Let $x\in P$. It is obvious that $\ell(x,(P\setminus X))$ is infinite if and only if $Q_x\cap (P\setminus X)=\emptyset$ or, equivalently, $Q_x\subseteq X$, and that $\ell(x,\zeta (P\setminus X))$ is infinite precisely when $Q_x\cap\zeta (P\setminus X) = \emptyset$, that is, since $\zeta$ is an involution,
$\zeta (Q_x)\cap (P\setminus X) = \emptyset$ or, equivalently, $\zeta (Q_x)\subseteq X$.\\

Now we prove, inductively, the first part. As $X^{0(\prime \ast)}=X$, it is obvious that $X^{0(\prime \ast)}=\{x\colon  \ell (x, (P\setminus X))>0\}  $.\\

Proceed inductively. For $n$ even, suppose that
 $X^{n(\prime \ast)}  = \{x\colon  \ell (x, (P\setminus X))>n\} $. Hence, we have
$P\setminus X^{n(\prime \ast)} = \{ x\colon \ell(x, (P\setminus X))\leq n \}$.
Since $\zeta$ is a dual order-isomorphism, we get 
$X^{n(\prime \ast)\prime}= \zeta ( P\setminus X^{n(\prime \ast)}) = \{ \zeta (x)\colon \ell(x, (P\setminus X))\leq n \}=\{ \zeta (x)\colon \ell(\zeta (x), \zeta (P\setminus X))\leq n \}=\{ x\colon \ell(x, \zeta (P\setminus X))\leq n  \}$. Thus,
\begin{equation} \label{eq}
X^{n(\prime \ast)\prime}= \{ x\colon \ell(x, \zeta (P\setminus X))\leq n  \}.
\end{equation} 

  By definition,   $X^{(n+1)(\prime \ast)}=P\setminus [X^{n(\prime \ast) \prime})$. Obviously,  $\{ x\colon \ell(x, \zeta (P\setminus X))> n+1  \}\subseteq X^{(n+1)(\prime \ast)}$,  in view of \eqref{eq}. Let $x\in X^{(n+1)(\prime \ast)}$.  So, we have, by \eqref{eq}, that $\ell(x,\zeta (P\setminus X))\not \leq n$. We also claim that $\ell(x,\zeta (P\setminus X))>n+1$.  For, suppose $\ell(x,\zeta (P\setminus X))=n+1$. Then $\ell(x, z)=n+1$, for some $z\in \zeta (P\setminus X)$.
As $n+1$ is odd, we have that $x\in \mathrm{Max}(P)$ and $z\in \mathrm{Min}(P)$, or $x\in \mathrm{Min}(P)$ and $z\in \mathrm{Max}(P)$. The latter yields a contradiction, since, as $\zeta (P\setminus X)$ is decreasing, there would exist $z_1\in \zeta (P\setminus X)$ such that $\ell(x,z_1)=n$ and so $\ell(x,\zeta (P\setminus X))\leq n$.   Thus $x\in \mathrm{Max}(P)$ and $z\in \mathrm{Min}(P)$ and then $x> y$, for some $y$ such that $\ell(y,\zeta (P\setminus X))=n$. So $x\in [X^{n(\prime \ast) \prime})$, a contradiction, proving the claim. Thus $X^{(n+1)(\prime \ast)}  \subseteq \{ x\colon \ell (x,\zeta (P\setminus X))>n+1  \} $ and we conclude that $X^{(n+1)(\prime \ast)}= \{ x\colon \ell (x,\zeta (P\setminus X))>n+1  \}$.\\

Now, as $X^{(n+1)(\prime \ast)}= \{ x\colon \ell (x,\zeta (P\setminus X))>n+1  \} $ and $\zeta$ is an involutive dual order-isomorphism, we have 
$X^{(n+1)(\prime \ast) \prime}=\zeta ( P\setminus X^{(n+1)(\prime \ast)}) = \{ \zeta (x)\colon \ell(x, \zeta (P\setminus X))\leq n+1 \}=\{ \zeta (x)\colon \ell(\zeta(x), (P\setminus X))\leq n+1 \}
 = \{ x\colon \ell(x, (P\setminus X))\leq n+1  \}$. By definition, $ X^{(n+2)(\prime \ast)}=P\setminus  [X^{(n+1)(\prime \ast) \prime})$. It is obvious that $\{ x\colon \ell (x,(P\setminus X))>n+2  \}\subseteq X^{(n+2)(\prime \ast)}$.
 Let $x\in X^{(n+2)(\prime \ast)}$. Then $\ell(x, (P\setminus X))> n+1$.  Suppose $\ell(x, (P\setminus X))= n+2$. Then $\ell(x, z)=n+2$, for some $z\in P\setminus X$. As $n+2$ is even, we have that $x, z\in \mathrm{Max}(P)$  or $x, z\in \mathrm{Min}(P)$. The latter yields a contradiction, since, as $P\setminus X$ is increasing, there would exist $z_1\in P\setminus X$ such that $\ell(x,z_1)=n+1$ and so $\ell(x, (P\setminus X))\leq n+1$. Thus $x, z\in \mathrm{Max}(P)$ and $x> y$, for some $y$ such that $\ell(y,(P\setminus X))=n+1$. So $x\in [X^{(n+1)(\prime \ast) \prime})$, which is a contradiction. Thus $X^{(n+2)(\prime \ast)}  \subseteq \{ x\colon \ell (x,(P\setminus X))>n+2  \} $ and we conclude that $X^{(n+2)(\prime \ast)}= \{ x\colon \ell (x, (P\setminus X))>n+2  \}$.
\end{proof}

Let $P$ be a $pm$-space  and $x,y\in P$.  The  {\it $\zeta$-distance  from $x$ to $y$}, denoted $\ell_{\zeta}(x,y)$, is the least 
 element of the set $\{\ell(x,y), \ell (x,\zeta (y))\}$, that is, $\ell_{\zeta}(x,y)=\mathrm{min}\{\ell(x,y), \ell (x,\zeta (y))\}$. 
Since $\ell (x,\zeta (y))=\ell(\zeta(x),y)$, we have that $\ell_{\zeta}(x,y)=\ell_{\zeta}(y,x)$ and we just call $\ell_{\zeta}(x,y)$ the {\it $\zeta$-distance between $x$ and $y$}. If $\ell_{\zeta}(x,y)=\infty$, we say that the $\zeta$-distance between $x$ and $y$ is {\it infinite}, otherwise it is said to be {\it finite}.\\

From the definition of $\zeta$-distance and from the fact that $\zeta$ is an involutive dual order-isomorphism, it is clear that  $\ell_{\zeta}(x,y)=\ell_{\zeta}(\zeta(x),\zeta (y))=\ell_{\zeta}(x,\zeta (y))=\ell_{\zeta}(\zeta(x),y)$.\\

Let $ n<\omega$. We say that $P$ is  of  {\it $\zeta$-width} $n$ if, for all $x,y\in P$, whenever it is the case  that $\ell_{\zeta}(x,y)$ is finite, then $\ell_{\zeta}(x,y)\leq n$. 
It is clear that $P$ is  of  $\zeta$-width $n$ if and only if, for any order component $Q\subseteq P$ and any $x,y\in Q$, $\ell_{\zeta}(x,y)\leq n$.\\

The immediate objective is to establish Theorem \ref{C3}, the next lemma is in preparation for this.

\begin{Lemma}\label{C3B} Let $P$ be a $pm$-space of height at most $1$, $X\subseteq P$ a clopen decreasing set and let $Q$ be an order component of $P$. We have\\
{\rm (1)} If $Q\cup \zeta(Q)\subseteq X$, then $Q\cup \zeta (Q)\subseteq (X\cap X^{\prime \ast})^{m(\prime \ast)}$, for every $m<\omega$.\\
{\rm (2)} If $(Q\cup \zeta (Q)) \cap X = \emptyset$, then  $(Q\cup \zeta (Q)) \cap (X\cap X^{\prime \ast})^{m(\prime \ast)} = \emptyset$, for every $m<\omega$.\\
{\rm (3)} If $(Q\cup \zeta (Q)) \cap X \neq \emptyset$, $(Q\cup \zeta (Q)) \cap (P\setminus X) \neq \emptyset$ and there exists $n<\omega$, such that, for any $x,y\in Q$, $\ell_{\zeta}(x,y)\leq n$, then $(Q\cup \zeta (Q)) \cap (X\cap X^{\prime \ast})^{m(\prime \ast)} = \emptyset$, for $n\leq m<\omega$.
\end{Lemma}

\begin{proof}
Recall that $\zeta$ is an involutive dual order-isomorpism.
Let $Q\subseteq P$ be an order component. Then $\zeta (Q)$ is also an order component.\\

(1) It is sufficient to show that for any clopen decreasing set $Y\subseteq P$ such that $Q\cup \zeta (Q) \subseteq Y$, we have $Q\cup \zeta (Q) \subseteq Y^{\prime\ast}$.  Let $x\in Q\cup \zeta (Q)$. We have that $Q_x=Q$ or $Q_x=\zeta(Q)$, so $Q_x\subseteq Q\cup \zeta (Q) \subseteq Y$ and, consequently, $\zeta(Q_x)\subseteq Q\cup \zeta (Q) \subseteq Y$. Applying Lemma \ref{C2}, we conclude first that $\ell (x,\zeta(P\setminus Y))$ is infinite and  then that $x\in Y^{\prime\ast}$. Thus   $Q\cup \zeta (Q)\subseteq  Y^{\prime\ast}$. \\

(2) Suppose that $(Q\cup \zeta (Q)) \cap X = \emptyset$.  Then  
 $(Q\cup \zeta (Q)) \cap (X\cap X^{\prime \ast}) = \emptyset$.
As, for any $m<\omega$, $(X\cap X^{\prime \ast})^{m(\prime \ast)}\subseteq X\cap X^{\prime \ast}$, we conclude that
$(Q\cup \zeta (Q)) \cap (X\cap X^{\prime \ast})^{m(\prime \ast)} = \emptyset$, for every $m<\omega$.\\

(3) Suppose that $(Q\cup \zeta (Q)) \cap X \neq\emptyset $, $(Q\cup \zeta (Q)) \cap (P\setminus X) \neq \emptyset$ and that there exists $n<\omega$, such that, for any $x,y\in Q$, $\ell_{\zeta}(x,y)\leq n$. As $\ell_{\zeta}(x,y)=\ell_{\zeta}(\zeta(x),\zeta (y))=\ell_{\zeta}(x,\zeta (y))=\ell_{\zeta}(\zeta(x),y)$, we have that $\ell_{\zeta}(x,y)\leq n$, for any $x,y\in Q\cup\zeta(Q)$. Take $y\in (Q\cup \zeta (Q)) \cap (P\setminus X)$. As $y\in P\setminus X$, we have 
$\zeta (y)\in \zeta (P\setminus X)$.  Applying Lemma \ref{C2}, we know that $\zeta (y)\not\in X^{\prime \ast}$, since $\ell (\zeta (y), \zeta (P\setminus X)) = 0$.
In particular  $y, \zeta (y)\not\in X\cap X^{\prime \ast}$.  That is  $y, \zeta (y)\in P\setminus  (X\cap X^{\prime \ast})$ and, likewise, $y, \zeta (y)\in \zeta (P\setminus  (X\cap X^{\prime \ast}))$. 
Let $x\in Q\cup \zeta (Q)$. As $y\in Q\cup \zeta (Q)$, we have  $\ell_{\zeta}(x,y)\leq n$, that is, $\ell (x,y)\leq n$ or $\ell (x,\zeta(y))\leq n$. As $y, \zeta (y)\in P\setminus  (X\cap X^{\prime \ast})$, we conclude that $\ell (x, (P\setminus  (X\cap X^{\prime \ast})))\leq n$ and, as  $y, \zeta (y)\in \zeta (P\setminus  (X\cap X^{\prime \ast}))$, we have that $\ell (x, \zeta (P\setminus  (X\cap X^{\prime \ast})))\leq n$. Applying Lemma \ref{C2},
$x\not\in (X\cap X^{\prime \ast})^{n(\prime \ast)}$.
That is, $(Q\cup \zeta (Q))\cap (X\cap X^{\prime \ast})^{n(\prime \ast)} = \emptyset$ and the result follows from the fact that, for any $m\geq n$, $(X\cap X^{\prime \ast})^{m(\prime \ast)}\subseteq (X\cap X^{\prime\ast})^{n(\prime \ast)}$.
\end{proof}

\begin{Theorem}\label{C3} Let $L$ be a regular $pm$-algebra and $P$ be its dual $pm$-space. If $P$ has $\zeta$-width $n<\omega$, then $L$ is of range $n$.

\end{Theorem}

\begin{proof} As $L$ is a regular $pm$-algebra, $P$ has height at most $1$, by Theorem \ref{C1}. Suppose the $\zeta$-width of $P$ is $n<\omega$. If $P=\emptyset$, then $L$ is trivial and is, obviously, of range $n$. Let $P\neq\emptyset$. We must show that, for any $x\in L$, $(x\wedge x^{\prime \ast})^{n(\prime \ast)} = (x\wedge x^{\prime \ast})^{(n+1)(\prime \ast)}$, or, equivalently, for any clopen decreasing set $X\subseteq P$, $(X\cap X^{\prime \ast})^{n(\prime \ast)} = (X\cap X^{\prime \ast})^{(n+1)(\prime \ast)}$. As the $\zeta$-width of $P$ is $n$, then, for any order component $Q\subseteq P$ and any $x,y\in Q$, $\ell_{\zeta}(x,y)\leq n$. Let $X\subseteq P$ be a clopen decreasing set. For any order component $Q$ we have that $Q\cup\zeta(Q)\subseteq X$, or $(Q\cup\zeta(Q))\cap X=\emptyset$, or $(Q\cup \zeta (Q)) \cap X \neq \emptyset$ and $(Q\cup \zeta (Q)) \cap (P\setminus X) \neq \emptyset$, so, applying, respectively, (1), (2) and (3) of the previous lemma \ref{C3B}, we conclude that $(Q\cup \zeta (Q))\cap (X\cap X^{\prime \ast})^{n(\prime \ast)}=(Q\cup \zeta (Q))\cap (X\cap X^{\prime \ast})^{(n+1)(\prime \ast)}$. Thus $(X\cap X^{\prime \ast})^{n(\prime \ast)}=(X\cap X^{\prime \ast})^{(n+1)(\prime \ast)}$, since the order components of $P$ form a partition of $P$.
\end{proof}

The next objective is the converse of Theorem \ref{C3}, for which (3) of the next lemma isolates the technical step needed.\\

Let $P$ be  $pm$-space. For each $n<\omega$, and each $x\in P$, let $D_n(x):  =\{z \in P \colon \ell(z,x)\leq n\}$.

\begin{Lemma}\label{C3A} Let $P$ be a $pm$-space of height at most $1$, $x, y\in P$ and $n<\omega$.  We have\\
{\rm (1)} If $x\in  \mathrm{Min}(P)$, then $D_n(x)$ is closed and decreasing whenever $n$ is even, and $D_n(x)$ is closed and increasing whenever $n$ is odd.\\
{\rm (2)} If $x\in  \mathrm{Max}(P)$, then  $D_n(x)$ is closed and increasing whenever $n$ is even, and $D_n(x)$ is closed and decreasing whenever $n$ is odd.\\
{\rm (3)} If $n\geq 1$, $\ell(x,y)=n$, $\ell(x, \zeta (y))\not\leq n$ and $D_{n-1}(x)\cup D_n(\zeta(x))$ is decreasing, then there exists a clopen decreasing set $X$ such that $\ell(x, (P\setminus (X\cap X^{\prime \ast})))=n$.
\end{Lemma}

\begin{proof}
(1) Suppose $x\in  \mathrm{Min}(P)$.  We have that $D_0(x)=\{z\colon \ell(z,x)=0\}=\{x\}$, which is closed and, since $x\in  \mathrm{Min}(P)$, it is decreasing. 
We proceed inductively. Suppose that $n$ is even and that $D_n(x)$ is closed and decreasing. We have $D_{n+1}(x)=\{z\colon \ell(z,x)\leq n+1\}$. Obviously, $[D_n(x))\subseteq D_{n+1}(x)$. Let $z\in D_{n+1}(x)$. Then, either $z\in D_n(x)$ or $\ell(z,x)=n+1$. In the latter case, as $n+1$ is odd and $x\in  \mathrm{Min}(P)$, we have that $z\in  \mathrm{Max}(P)$ and so $z>z_1$, for some $z_1\in D_n(x)$. Thus, $z\in [D_n(x))$. Consequently, $D_{n+1}(x)=[D_n(x))$, which is increasing and closed, since $D_n(x)$ is closed. Now, it is obvious that $(D_{n+1}(x)]\subseteq D_{n+2}(x)$. Let $z\in D_{n+2}(x)$. Then, either $z\in D_{n+1}(x)$ or $\ell(z,x)=n+2$. In the latter case, as $n+2$ is even and $x\in  \mathrm{Min}(P)$, we have that $z\in  \mathrm{Min}(P)$ and so $z<z_1$, for some $z_1\in D_{n+1}(x)$. Thus, $z\in (D_{n+1}(x)]$. Consequently, $D_{n+2}(x)=(D_{n+1}(x)]$, which is decreasing and closed, since $D_{n+1}(x)$ is closed.\\

(2) Suppose $x\in  \mathrm{Max}(P)$. Then $\zeta(x)\in  \mathrm{Min}(P)$. As $\zeta$ is an involutive continuous dual order-isomorphism $\zeta (D_n(\zeta(x)))=D_n(x)$, for any $n<\omega $, and $\zeta (D_n(\zeta(x)))$ is closed and decreasing (respectively, increasing) if and only if  $D_n(\zeta(x))$ is closed and increasing (respectively, decreasing). Now we just apply (1).\\

(3) Suppose $n\geq 1$, $\ell(x,y)=n$, $\ell(x, \zeta (y))\not\leq n$ and $D_{n-1}(x)\cup D_n(\zeta(x))$ is decreasing. By (1) and (2), we know that $D_{n-1}(x)\cup D_n(\zeta(x))$ is closed. As $\ell(x,y)=n$ and $\ell(y, \zeta (x))=\ell(x, \zeta (y))\not\leq n$, we have that $y\notin D_{n-1}(x)\cup D_n(\zeta(x))$. So, the set $D_{n-1}(x)\cup D_{n}(\zeta(x))$ is closed and decreasing and $y\notin D_{n-1}(x)\cup D_{n}(\zeta(x))$. It follows that there exists a clopen increasing set $Y$ such that $y\in Y$ and $(D_{n-1}(x)\cup D_{n}(\zeta(x)))\cap Y=\emptyset$. 
Let $X=P\setminus Y$, which is a clopen decreasing set such that $D_{n-1}(x)\cup D_{n}(\zeta(x))\subseteq X$ and $y\notin X$. From $D_{n}(\zeta(x))\cap Y=\emptyset$, we get $\zeta(D_{n}(\zeta(x)))\cap \zeta(Y)=\emptyset$. Therefore, as $\zeta(D_{n}(\zeta(x)))=D_{n}(x)$, it follows that $D_{n}(x)\cap \zeta(Y)=\emptyset$. Suppose $D_{n-1}(x)\not\subseteq X^{\prime \ast}$, that is, there exists $z\in D_{n-1}(x)$ such that $z\notin X^{\prime \ast}$. For such a $z$ we have $\ell(z,x)\leq n-1$ since $z\in D_{n-1}(x)$ and, applying Lemma \ref{C2}, $\ell(z,\zeta(Y))\leq 1$ since $z\notin X^{\prime \ast}$. Consequently, $\ell(x,\zeta(Y))\leq n$ and $D_n(x)\cap \zeta(Y)\neq\emptyset$, a contradiction. Thus  $D_{n-1}(x)\subseteq X^{\prime \ast}$ and so  $D_{n-1}(x)\subseteq X\cap X^{\prime \ast}$. As $y\notin X\cap X^{\prime \ast}$ and $\ell(x,y)=n$, we have  $\ell(x, (P\setminus (X\cap X^{\prime \ast})))\leq n$.   But $D_{n-1}(x)\subseteq X\cap X^{\prime \ast}$ and, consequently, $\ell(x, ( P\setminus (X\cap X^{\prime \ast})))> n-1$.  So $\ell(x, ( P\setminus (X\cap X^{\prime \ast})))=n$, as required.
\end{proof}

\begin{Theorem}\label{C4} Let $L$ be a regular $pm$-algebra and $P$ be its dual $pm$-space.
If there exist $x,y\in P$ such that $\ell_{\zeta} (x,y) = n$
for some  $1\leq n <\omega$, then $L$ does not satisfy the identity $(x\wedge x^{\prime \ast})^{(n-1)(\prime \ast)} \approx (x\wedge x^{\prime \ast})^{n(\prime \ast)}$.
\end{Theorem}

\begin{proof} 
As $L$ is a regular $pm$-algebra, by Theorem \ref{C1}, $P$ has height at most $1$. Suppose $\ell_{\zeta} (x,y) = n$ for some $x,y\in P$ and $n\geq 1$. We begin with two small observations that will simplify the subsequent argument. First, we recall that $\ell_{\zeta} (x,y)=\ell_{\zeta}(\zeta(x), \zeta (y))$. Second, suppose that $\ell_{\zeta} (x,y)= \ell (x,\zeta (y))$. Since  $\ell_{\zeta}(\zeta(x), y)=\ell_{\zeta} (x,y)$, we have $\ell_{\zeta}(\zeta(x), y)=\ell (x,\zeta (y))=\ell(\zeta(x), y)$. Thus, with no loss of generality, we will assume, by the first observation, that $y\in\mathrm{Max}(P)$ and, by the second, that $\ell_{\zeta} (x,y)=\ell (x,y)$.\\

Suppose $n$ is odd.  Since $y\in \mathrm{Max}(P)$ and $n=\ell_{\zeta} (x,y) = \ell (x,y)$, we have that $x\in \mathrm{Min}(P)$ and so $\zeta(x)\in \mathrm{Max}(P)$.
We know, by (1) and (2) of Lemma \ref{C3A}, that both $D_{n-1}(x)$ and $D_{n}(\zeta(x))$ are closed and decreasing and so $D_{n-1}(x)\cup D_n(\zeta(x))$ is closed and decreasing. Since  $\ell_{\zeta} (x,y)=n$, we have that $\ell(x, \zeta (y))\not\leq n-1$. As $x, \zeta(y)\in \mathrm{Min}(P)$ and $n$ is odd we cannot have $\ell(x, \zeta (y))=n$. So $\ell(x, \zeta (y))\not\leq n$. Applying (3) of Lemma \ref{C3A},
there exists a clopen decreasing set $X$ such that
 $\ell(x, ( P\setminus (X\cap X^{\prime \ast})))=n$. For such an $X$, applying Lemma \ref{C2}, we have that 
$x\in (X\cap X^{\prime \ast})^{(n-1)(\prime \ast)}$, whilst
$x\not\in (X\cap X^{\prime \ast})^{(n+1)(\prime \ast)}$.
In particular, $(X\cap X^{\prime \ast})^{(n-1)(\prime \ast)}\neq (X\cap X^{\prime \ast})^{(n+1)(\prime \ast)}$.
It follows that $(X\cap X^{\prime \ast})^{(n-1)(\prime \ast)}\neq (X\cap X^{\prime \ast})^{n(\prime \ast)}$,
that is $L$ does not satisfy the identity $(x\wedge x^{\prime \ast})^{(n-1)(\prime \ast)} \approx (x\wedge x^{\prime \ast})^{n(\prime \ast)}$.\\

Now suppose $n$ is even. Since $n=\ell_{\zeta} (x,y) = \ell (x,y)$ and $y\in \mathrm{Max}(P)$, we have that $x\in \mathrm{Max}(P)$ and so $\zeta(x)\in \mathrm{Min}(P)$. We know, by (1) and (2) of Lemma \ref{C3A}, that both $D_{n-1}(x)$ and $D_{n}(\zeta(x))$ are closed and decreasing and so $D_{n-1}(x)\cup D_n(\zeta(x))$ is closed and decreasing. Since $\ell_{\zeta} (x,y)=n$, we have that $\ell(x, \zeta (y))\not\leq n-1$. As $x\in \mathrm{Max}(P)$ and $\zeta(y)\in \mathrm{Min}(P)$ and $n$ is even we cannot have $\ell(x, \zeta (y))=n$. So $\ell(x, \zeta (y))\not\leq n$. Applying (3) of Lemma \ref{C3A},
there exists a clopen decreasing set $X$ such that
 $\ell(x, ( P\setminus (X\cap X^{\prime \ast})))=n$. For such an $X$, we have $\ell(\zeta(x), \zeta ( P\setminus (X\cap X^{\prime \ast})))=n$. Therefore, by Lemma \ref{C2}, $\zeta (x)\in (X\cap X^{\prime \ast})^{(n-1)(\prime \ast)}$, whilst
$\zeta (x)\not\in (X\cap X^{\prime \ast})^{(n+1)(\prime \ast)}$.
In particular, $(X\cap X^{\prime \ast})^{(n-1)(\prime \ast)}\neq (X\cap X^{\prime \ast})^{(n+1)(\prime \ast)}$.
It follows that $(X\cap X^{\prime \ast})^{(n-1)(\prime \ast)}\neq (X\cap X^{\prime \ast})^{n(\prime \ast)}$,
that is $L$ does not satisfy the identity $(x\wedge x^{\prime \ast})^{(n-1)(\prime \ast)} \approx (x\wedge x^{\prime \ast})^{n(\prime \ast)}$.
\end{proof}

\begin{Corollary} \label{C4B} Let $L$ be a regular $pm$-algebra and $P$ be its dual $pm$-space. Then $L$ is of range $n<\omega$ if and only if   $P$ is of $\zeta$-width $n$. 
\end{Corollary}
\begin{proof} Suppose $L$ is of range $n$. Suppose $x,y\in P$ such that $\ell_{\zeta}(x,y)$ is finite. If $\ell_{\zeta}(x,y)=k>n$, then, by Theorem \ref{C4}, $L$ does not satisfy the identity $(x\wedge x^{\prime \ast})^{(k-1)(\prime \ast)} \approx (x\wedge x^{\prime \ast})^{k(\prime \ast)}$ and, as $k-1\geq n$, $L$ is not of range $n$, a contradiction. So, $\ell_{\zeta}(x,y)\leq n$. Thus, $P$ is of $\zeta$-width $n$. For the converse just apply Theorem \ref{C3}.
\end{proof}

\section{Subdirectly irreducible algebras and local finiteness}\label{SD}

As we have already mentioned, for each $n<\omega$, the variety $\mathbf{M}_n$ of regular $pm$-algebras of range $n$ is a discriminator variety and, consequently, its sudirectly irreducible algebras are the simple ones. We begin by characterising the simple algebras of $\mathbf{M}_n$, $n<\omega$, via their dual spaces. For this purpose, the following theorem will be useful.\\

\begin{Theorem} \label{X4} Let $L\in \mathbf{M}_n$ and $P$ be its dual $pm$-space. Then all order components of $P$ are closed.
\end{Theorem}
\begin{proof}
Suppose that $L\in \mathbf{M}_n$. Then, by Corollary \ref{C4B}, $P$ is of $\zeta$-width $n$. So, for all $x,y\in P$ such that $\ell_{\zeta}(x,y)$ is finite, we have $\ell_{\zeta}(x,y)\leq n$.\\

Let $Q$ be an order component of $P$ and $x\in Q$. Obviously, $D_n(x)\subseteq Q$. We know that $Q\cap\zeta(Q)=\emptyset$ or $Q= \zeta(Q)$.  So, we have two cases to consider.  Firstly,  suppose that $Q\cap\zeta(Q)=\emptyset$. Let $y\in Q$. We have that $\ell(x,y)$ is finite and $\ell(x,\zeta(y))$ is infinite, so $\ell(x,y)=\ell_{\zeta}(x,y)\leq n$ and, consequently, $y\in D_n(x)$. Thus, $Q=D_n(x)$, which is closed in $P$, by (1) and (2) of Lemma \ref{C3A}. Secondly, suppose that $Q= \zeta(Q)$. Consider $D_n(x)$ and $D_n(\zeta(x))$. Obviously, $D_n(x)\subseteq Q$ and, since $\zeta(x)\in Q$, we also have $D_n(\zeta(x))\subseteq Q$. Thus, $D_n(x)\cup D_n(\zeta(x))\subseteq Q$.
Let $y\in Q$. As  $Q= \zeta(Q)$, we have that both $\ell(y,x)$ and $\ell(y,\zeta(x))$ are finite and, since $\ell_{\zeta}(y,x)\leq n$,  at least, one of them does not exceed $n$. Thus, $y\in D_n(x)\cup D_n(\zeta(x))$ and we have $Q=D_n(x)\cup D_n(\zeta(x))$. By (1) and (2) of Lemma \ref{C3A}, $Q$ is closed in $P$. 
\end{proof}

\begin{Theorem} \label{X5}  Let $L\in \mathbf{M}_n$ and $P$ be its dual $pm$-space.
Then the following are equivalent:\\
{\rm (1)} $L$ is simple,\\
{\rm (2)}  $P=Q\cup\zeta (Q)$ where $Q$ is an order component of $P$. 
\end{Theorem}

\begin{proof}
As we have already observed in Section \ref{SB} there is a lattice-isomorphism from the lattice of congruences on $L$ to the lattice of open sets $X$ of $P$ such that $\zeta(X)=X$ and $[{\rm Min}(P)\cap X)\subseteq X$. Moreover any such open set also satisfies $({\rm Max}(P)\cap X]\subseteq X$.\\

Suppose that $L$ is simple. Consider $Q\cup\zeta(Q)$, where $Q$ is an order component of $P$. Suppose $P\neq Q\cup\zeta(Q)$. Let $X=P\setminus (Q\cup\zeta(Q))$.  We wish to show that $X$ represents a congruence on $P$.  Observe that $X$ is an open set, 
since both $Q$ and $\zeta(Q)$ are order components and hence closed, by the previous theorem. Obviously, $\zeta(X)=X$. Let  $y\in [{\rm Min}(P)\cap X)$. Then $y\geq x$, for some $x\in {\rm Min}(P)\cap X$. If $y\in Q\cup\zeta (Q)$, so would $x$, since $Q$ and $\zeta (Q)$ are order components. So $y\in P\setminus (Q\cup\zeta(Q))= X$. Thus $[{\rm Min}(P)\cap X)\subseteq X$ and $X$ represents a congruence on $L$. Since $X\neq\emptyset$ and $X\neq P$, we have a contradiction to the fact that $L$ is simple.\\

Conversely, suppose that $P=Q\cup\zeta (Q)$ for an order component $Q$. Let $\theta$ be a congruence on $L$ different from the equality congruence and let $X$ be the open set that represents $\theta$. Then $X\neq \emptyset$, $\zeta(X)=X$, $[{\rm Min}(P)\cap X)\subseteq X$ and $({\rm Max}(P)\cap X]\subseteq X$.  Since $P={\rm Min}(P)\cup {\rm Max}(P)$, we conclude that $X$ is increasing and decreasing. From $\zeta(X)=X\neq\emptyset$ and $P=Q\cup\zeta(Q)$, we have  that $\emptyset\neq X=(Q\cap X)\cup (\zeta(Q)\cap X) = (Q\cap X)\cup (\zeta(Q)\cap \zeta(X))=(Q\cap X)\cup \zeta(Q\cap X)$  and hence $Q\cap X\neq\emptyset$. Thus $Q\subseteq X$, since $X$ is both increasing and decreasing, 
and we conclude that $P= Q\cup\zeta(Q)\subseteq X\cup\zeta(X)=X$. So $\theta$ is the universal congruence. 
\end{proof}

\begin{Corollary}\label{X6} Let $L$ be a non-trivial regular $pm$-algebra and $P$ be its dual $pm$-space. Then $L$ is a simple algebra of $\mathbf{M}_n$ if and only if $\ell(x,y)\leq n$ or $\ell(x,\zeta(y))\leq n$, for all $x,y\in P$.
\end{Corollary}
\begin{proof} Suppose $L$ is a simple algebra of $\mathbf{M}_n$. Then, by Corollary \ref{C4B} and Theorem \ref{X5}, $P$ is of $\zeta$-width $n$ and $P=Q\cup\zeta(Q)$ with $Q$ order component. The set $\zeta(Q)$ is also an order component. Let $x,y\in P$. If $x,y\in Q$ or $x,y\in\zeta (Q)$, then $\ell(x,y)$ is finite. Otherwise $\ell(x,\zeta(y))$ is finite, since both elements will belong to the same order component. Either way $\ell_{\zeta}(x,y)$ is finite and then, as $P$ is of $\zeta$-width $n$, $\ell_{\zeta}(x,y)\leq n$, that is $\ell(x,y)\leq n$ or $\ell(x,\zeta(y))\leq n$.\\

Conversely, suppose that, for all $x,y\in P$, $\ell(x,y)\leq n$ or $\ell(x,\zeta(y))\leq n$. It is obvious that $P$ is of $\zeta$-width $n$ and then $L\in\mathbf{M}_n$, by Corollary \ref{C4B}. Let $x\in P$ and let $Q_x$ be the order component such that $x\in Q_x$. For any $y\in P$, we have $\ell(x,y)\leq n$ or $\ell(x,\zeta(y))\leq n$ and, consequently, $y\in Q_x$ or $y\in\zeta(Q_x)$. Thus $P=Q_x\cup\zeta(Q_x)$ and $L$ is simple, by the previous theorem.
\end{proof}

In Figure $1$, clouds denote order components.\\

\begin{figure}[ht]
\begin{center}
\unitlength 1mm 
\linethickness{0.4pt}
\ifx\plotpoint\undefined\newsavebox{\plotpoint}\fi 
\begin{picture}(141.25,119.75)(0,0)
\put(24,110.25){\circle*{2.5}}
\put(60.25,110.5){\circle*{2.5}}
\put(122.25,115.75){\circle*{2.5}}
\put(12.25,79){\circle*{2.5}}
\put(67.5,79.25){\circle*{2.5}}
\put(48.25,34.75){\circle*{2.5}}
\put(25.5,35){\circle*{2.5}}
\put(117.75,79.75){\circle*{2.5}}
\put(112.5,35.5){\circle*{2.5}}
\put(89.75,35.25){\circle*{2.5}}
\put(33.25,78.75){\circle*{2.5}}
\put(78.75,79){\circle*{2.5}}
\put(59.5,34.5){\circle*{2.5}}
\put(36.75,34.75){\circle*{2.5}}
\put(129,79.5){\circle*{2.5}}
\put(123.75,35.25){\circle*{2.5}}
\put(101,35.5){\circle*{2.5}}
\put(81.5,110.5){\circle*{2.5}}
\put(122.25,104){\circle*{2.5}}
\put(12.25,67.25){\circle*{2.5}}
\put(67.5,67.5){\circle*{2.5}}
\put(48.25,23){\circle*{2.5}}
\put(25.5,23.25){\circle*{2.5}}
\put(117.75,68){\circle*{2.5}}
\put(112.5,23.75){\circle*{2.5}}
\put(89.75,23.5){\circle*{2.5}}
\put(33.25,67){\circle*{2.5}}
\put(78.75,67.25){\circle*{2.5}}
\put(59.5,22.75){\circle*{2.5}}
\put(36.75,23){\circle*{2.5}}
\put(129,67.75){\circle*{2.5}}
\put(123.75,23.5){\circle*{2.5}}
\put(101,23.75){\circle*{2.5}}
\put(23.625,109.875){\oval(19.25,15.25)[]}
\put(59.875,110.125){\oval(19.25,15.25)[]}
\put(81.125,110.125){\oval(19.25,15.25)[]}
\put(55.5,110.25){\makebox(0,0)[cc]{$x$}}
\put(117.5,115.5){\makebox(0,0)[cc]{$\zeta (x)$}}
\put(7.5,78.75){\makebox(0,0)[cc]{$\zeta(y)$}}
\put(62.75,79){\makebox(0,0)[cc]{$\zeta(y)$}}
\put(113,79.5){\makebox(0,0)[cc]{$\zeta (x)$}}
\put(28.5,78.5){\makebox(0,0)[cc]{$\zeta (x)$}}
\put(83.5,78.75){\makebox(0,0)[cc]{$\zeta(x)$}}
\put(133.75,79.25){\makebox(0,0)[cc]{$\zeta(y)$}}
\put(76.75,110.25){\makebox(0,0)[cc]{$\zeta (x)$}}
\put(117.5,103.75){\makebox(0,0)[cc]{$x$}}
\put(7.5,67){\makebox(0,0)[cc]{$x$}}
\put(62.75,67.25){\makebox(0,0)[cc]{$x$}}
\put(113,67.75){\makebox(0,0)[cc]{$x$}}
\put(28.5,66.75){\makebox(0,0)[cc]{$y$}}
\put(83.5,67){\makebox(0,0)[cc]{$y$}}
\put(133.75,67.5){\makebox(0,0)[cc]{$y$}}
\multiput(122,116.5)(.03125,-1.53125){8}{\line(0,-1){1.53125}}
\multiput(12,79.75)(.03125,-1.53125){8}{\line(0,-1){1.53125}}
\multiput(67.25,80)(.03125,-1.53125){8}{\line(0,-1){1.53125}}
\multiput(117.5,80.5)(.03125,-1.53125){8}{\line(0,-1){1.53125}}
\multiput(33,79.5)(.03125,-1.53125){8}{\line(0,-1){1.53125}}
\multiput(78.5,79.75)(.03125,-1.53125){8}{\line(0,-1){1.53125}}
\multiput(128.75,80.25)(.03125,-1.53125){8}{\line(0,-1){1.53125}}
\put(122,108.75){\oval(17,22)[]}
\put(12,72){\oval(17,22)[]}
\put(33,71.75){\oval(17,22)[]}
\put(73.25,71.625){\oval(35.5,25.25)[]}
\put(123.5,72.125){\oval(35.5,25.25)[]}
\put(67,79.25){\line(1,-1){11.75}}
\multiput(118,80.25)(.0336538462,-.0384615385){312}{\line(0,-1){.0384615385}}
\multiput(128.75,80.25)(-.0336391437,-.0382262997){327}{\line(0,-1){.0382262997}}
\put(74.25,34.75){\makebox(0,0)[cc]{$\ldots$}}
\put(74.75,22.25){\makebox(0,0)[cc]{$\ldots$}}
\multiput(25.5,35)(.0336826347,-.0366766467){334}{\line(0,-1){.0366766467}}
\multiput(36.75,22.75)(.0337243402,.0359237537){341}{\line(0,1){.0359237537}}
\multiput(48.25,35)(.0336391437,-.0359327217){327}{\line(0,-1){.0359327217}}
\multiput(59.25,23.25)(-.0646067416,.0337078652){356}{\line(-1,0){.0646067416}}
\multiput(36.25,35.25)(-.0336990596,-.0368338558){319}{\line(0,-1){.0368338558}}
\multiput(25.5,23.5)(.0659025788,.0336676218){349}{\line(1,0){.0659025788}}
\multiput(59.5,34.5)(-.0336826347,-.0344311377){334}{\line(0,-1){.0344311377}}
\put(48.25,23){\line(-1,1){12}}
\multiput(36.25,23.25)(.0681818182,.0337243402){341}{\line(1,0){.0681818182}}
\multiput(59.5,34.75)(-.0995702006,-.0336676218){349}{\line(-1,0){.0995702006}}
\multiput(89.25,35.25)(.03125,-1.4375){8}{\line(0,-1){1.4375}}
\multiput(89.5,23.75)(.0337243402,.0359237537){341}{\line(0,1){.0359237537}}
\multiput(101.25,24.25)(.0336826347,.0351796407){334}{\line(0,1){.0351796407}}
\multiput(112.5,36)(-.03125,-1.46875){8}{\line(0,-1){1.46875}}
\put(118.75,24.25){\line(0,-1){.5}}
\multiput(112.25,23.75)(.0336826347,.0351796407){334}{\line(0,1){.0351796407}}
\multiput(123.5,35.5)(-.03125,-1.40625){8}{\line(0,-1){1.40625}}
\multiput(123.25,24.25)(-.0336990596,.0352664577){319}{\line(0,1){.0352664577}}
\multiput(112.5,35.5)(-.0659025788,-.0336676218){349}{\line(-1,0){.0659025788}}
\multiput(89.5,23.75)(.0981375358,.0336676218){349}{\line(1,0){.0981375358}}
\multiput(123.75,35.5)(-.0651862464,-.0336676218){349}{\line(-1,0){.0651862464}}
\multiput(101,23.75)(-.0337243402,.0351906158){341}{\line(0,1){.0351906158}}
\multiput(89.5,35.75)(.0653089888,-.0337078652){356}{\line(1,0){.0653089888}}
\multiput(112.75,23.75)(-.0337078652,.0344101124){356}{\line(0,1){.0344101124}}
\multiput(100.75,36)(.0639044944,-.0337078652){356}{\line(1,0){.0639044944}}
\multiput(123.5,24)(-.0962078652,.0337078652){356}{\line(-1,0){.0962078652}}
\multiput(89.25,36)(-.043269231,-.033653846){104}{\line(-1,0){.043269231}}
\put(91.25,32.5){\line(0,1){0}}
\multiput(84.75,27)(.05487805,-.03353659){82}{\line(1,0){.05487805}}
\multiput(59.25,35)(.033505155,-.036082474){97}{\line(0,-1){.036082474}}
\multiput(59,23)(.036082474,.033505155){97}{\line(1,0){.036082474}}
\multiput(54.75,23.5)(-.05,-.0333333){15}{\line(-1,0){.05}}
\multiput(47.5,23)(-.0599730458,.0336927224){371}{\line(-1,0){.0599730458}}
\multiput(25.25,35.5)(.0906084656,-.0337301587){378}{\line(1,0){.0906084656}}
\put(73.875,28.5){\oval(119.25,26)[]}
\multiput(101,35.75)(-.03125,-1.53125){8}{\line(0,-1){1.53125}}
\put(42.5,38.80){\makebox(0,0)[cc]{$\zeta(I)$}}
\put(107,38.80){\makebox(0,0)[cc]{$\zeta(S\setminus I)$}}
\put(107,19){\makebox(0,0)[cc]{$S\setminus I$}}
\put(42.25,19){\makebox(0,0)[cc]{$I$}}
\put(23,93.25){\makebox(0,0)[cc]{$Q_0$}}
\put(71,93.5){\makebox(0,0)[cc]{$Q_1$}}
\put(122.5,94){\makebox(0,0)[cc]{$Q_2$}}
\put(22.25,51.5){\makebox(0,0)[cc]{$Q_3$}}
\put(73.75,52){\makebox(0,0)[cc]{$Q_4$}}
\put(123.75,52.75){\makebox(0,0)[cc]{$Q_5$}}
\put(73,8.5){\makebox(0,0)[cc]{$Q_6(I,S)$}}
\put(73,-5.5){\makebox(0,0)[cc]{Figure 1}}
\end{picture}

\end{center}
\label{fig 1}
\end{figure}

 The next theorem gives a description, via their dual spaces, of the non-trivial algebras in $\mathbf{M}_0$, leading to the description of its simple algebras presented in Corollary \ref{D1A}.

\begin{Theorem}\label{D1} Let $L$ be a non-trivial $pm$-algebra and $P$ be its dual $pm$-space. Then $L\in \mathbf{M}_0$ if and only if, for every order component $Q\subseteq P$, $Q\cup \zeta (Q)$ is a copy of one of $Q_0$, $Q_1$, or $Q_2$ {\rm (}as diagrammed in {\rm Figure $1$)}.
\end{Theorem}

\begin{proof} Applying Theorem \ref{C1} and Corollary \ref{C4B}, we know that  $L\in \mathbf{M}_0$ if and only $P$ has height at most $1$ and, for any order component $Q$ of $P$ and any $x,y\in Q$, we have $\ell_{\zeta} (x,y) = 0$.
 Suppose $L\in \mathbf{M}_0$. Let $Q$ be an order component of $P$ and $x\in Q$. If $y\in Q$ and $y\neq x$, then, as $\ell_{\zeta} (y,x) = 0$, we have that $\ell (y,\zeta(x)) = 0$, that is $y=\zeta(x)$. So $|Q|\leq 2$, since $\zeta$ is a bijection. Now it is clear that  $Q\cup \zeta (Q)$ is a copy of of $Q_0$ or $Q_1$, if $|Q|=1$, whilst $Q\cup \zeta (Q)$ is a copy of  $Q_2$, if $|Q|=2$.
The converse is straightforward.
 \end{proof}

In \cite[Theorem 6.5]{Sa87a}, Sankappanavar described, up to isomorphism, the $6$ subdirectly irreducible pseudocomplemented de Morgan algebras of range $0$. Exactly $3$ of them are regular. So  $\mathbf{M}_0$ has, up to isomorphism, $3$ simple algebras. Their dual spaces are copies of $Q_0$, $Q_1$, or $Q_2$ as diagrammed in Figure $1$. This fact is also an immediate consequence of Theorems  \ref{X5} and \ref{D1}.

\begin{Corollary}\label{D1A} The simple algebras in $\mathbf{M}_0$ are the ones whose dual $pm$-space is a copy of $Q_0$, $Q_1$, or $Q_2$ {\rm (}as diagrammed in {\rm Figure $1$)}.\hfill{$\Box$}
\end{Corollary}

Notice that the simple algebras in $\mathbf{M}_0$ are just the simple de Morgan algebras with the inherent pseudocomplementation.\\

For a Stone space $(S;\tau )$  and a subset $I$ of the set of its isolated points, let $Q_6(I,S)$ be the quadruple $(S\cup \zeta (S);\tau,\leq,\zeta )$ where $3\leq |S|$, $\zeta (S)$ denotes a homeomorphic copy of $S$, $\tau$ is the union topology, $\zeta (\zeta (x)) = x$ for $x\in S$, and $\leq$ is the partial order on $S\cup \zeta (S)$ induced by, for $x,y\in S$, \[x<\zeta(y) \mbox{ iff } x\neq y \mbox{ or } x\notin I.\]

\begin{Theorem}\label{D1C}  The quadruple $Q_6(I,S)=(S\cup \zeta (S);\tau,\leq,\zeta )$ is a $pm$-space of height $1$.
\end{Theorem}

\begin{proof}
It is immediate that $(S\cup \zeta (S); \leq)$ has height $1$ and that $(S\cup \zeta (S);\tau)$ is a compact Hausdorff space. Moreover, whenever $z$ is an isolated point of the Stone space $S$ or of the Stone space $\zeta(S)$, $z$ is also an isolated point of $Q_6(I,S)$ and, consequently, $\{z\}$ is clopen in $Q_6(I,S)$ and so are $S\setminus \{z\}$ and $\zeta(S)\setminus \{z\}$.\\

Let $x,y\in S\cup \zeta (S)$ be such that $x\not\leq y$.  First, suppose $x\in S$.  Then either $y\in S$ and there is a set $X\subseteq S$ clopen in $S$, and so clopen in $Q_6(I,S)$ and necessarily decreasing,  such that $y\in X$ and $x\not\in X$, or else $y=\zeta(x)$ with $x\in I$. In this latter case, namely $y=\zeta(x)$, $S\setminus \{x\}$ and $\{\zeta(x)\}$ are clopen in $Q_6(I,S)$, since $x$ and $\zeta(x)$ are isolated points. Thus $\{\zeta(x)\}\cup (S\setminus \{x\})$ is clopen decreasing  in $Q_6(I,S)$ and  $x\notin\{\zeta(x)\}\cup (S\setminus \{x\})$.  Next, suppose $x\in\zeta(S)$. Then, since $S$ is clopen decreasing in $Q_6(I,S)$ and $x\notin S$, the case $y\in S$ is obvious, and if $y\in \zeta(S)$, there is a set $X\subseteq \zeta(S)$ clopen in $\zeta(S)$, and so in $Q_6(I,S)$, such that $y\in X$ and $x\not\in X$ from which $y\in X\cup S$, $x\not\in X\cup S$ and $X\cup S$ is a clopen decreasing set of $Q_6(I,S)$. \\

Thus, $(S\cup \zeta (S);\tau,\leq )$ is a Priestley space and it is clear that $\zeta$ is a continuous order-reversing involution.\\

Suppose $X\subseteq S\cup \zeta (S)$ is a clopen decreasing set.  Obviously, if $X=\emptyset$, then $[X)=\emptyset$. If $X = \{ x\}$ for some $x\in I$, then $[X) = \{  x\}\cup (\zeta(S)\setminus \{ \zeta (x)\} )$ which is clopen in $Q_6(I,S)$, since $x$ and $\zeta(x)$ are isolated points.
In the remaining cases for $X$, if $|X\cap S|\geq 2$, or $X=\{x\}$ for some $x\in S\setminus I$, it is obvious that
$[X) = X\cup \zeta(S)$. Observe that whenever $X\cap \zeta(S)\neq \emptyset$, then, as $X$ is decreasing, $|X\cap S|\geq 2$, and, as already argued, $[X) = X\cup \zeta(S)$. So whatever the case, we always have $[X) = X\cup \zeta(S)$, which is clearly clopen in $Q_6(I,S)$.   
\\

Therefore $(S\cup \zeta (S);\tau,\leq,\zeta )$ is a $pm$-space.
\end{proof}

Our next goal is to present a description, via their dual spaces, of the non-trivial $pm$-algebras in  $\mathbf{M}_1$, which will lead to the description of its simple algebras in Corollary \ref{D2A}.

\begin{Lemma}\label{D1B} Let $P$ be a $pm$-space of height at most $1$ and $Q\subseteq P$ be an order component. Then\\
{\rm (1)} $\mathrm{Min}(Q)=\mathrm{Min}(P)\cap Q$ and it is the set of the minimal elements of $Q$. \\
 {\rm (2)} $\mathrm{Max}(Q)=\mathrm{Max}(P)\cap Q$ and it is the set of the maximal elements of $Q$.\\
 {\rm (3)} $Q=\mathrm{Min}(Q)\cup \mathrm{Max}(Q)$, and $\mathrm{Min}(Q)\cap\mathrm{Max}(Q)\neq \emptyset$ if and only if $|Q| = 1$.\\
 {\rm (4)} If $|Q|>1$, then $x\neq\zeta(x)$, for any $x\in Q$.\\
 {\rm (5)} If $P$ is of $\zeta$-width $1$, then $|\mathrm{Min}(Q)|=|\mathrm{Max}(Q)|$ and, if $|\mathrm{Min}(Q)|=|\mathrm{Max}(Q)|>1$, then $Q=\zeta(Q)$.\\
 {\rm (6)} If $P$ is of $\zeta$-width $1$ and $|Q|>2$, then $Q=\zeta(Q)$, $|\mathrm{Min}(Q)|\geq 2$ and, for distinct $x,y$ in $\mathrm{Min}(Q)$, $x<\zeta(y)$.\\
 
 \end{Lemma}
\begin{proof} 
(1)-(2) By definition  $\mathrm{Min}(Q)=\mathrm{Min}(P)\cap (Q]$ and $\mathrm{Max}(Q)=\mathrm{Max}(P)\cap [Q)$ and the result follows immediately from the fact that  $Q$ is an order component.\\

 (3) Immediate from (1) and (2) and the facts that $P$ is of height at most $1$ and $Q$ is an order component.  \\
 
(4) If $x\in Q$ is such that $x=\zeta(x)$, then $x\in\mathrm{Min}(Q)\cap\mathrm{Max}(Q)$ and, by (3), $|Q|=1$.\\

(5) Suppose $P$ is of $\zeta$-width $1$. We know that $|\mathrm{Min}(Q)|\geq 1$ and $|\mathrm{Max}(Q)|\geq 1$. Suppose $|\mathrm{Min}(Q)|> 1$ or $|\mathrm{Max}(Q)|> 1$. Then there exist $x,y\in Q$ such that $\ell (x,y)\ge 2$.  Since $\ell_{\zeta} (x,y)\le 1$, it follows that $\ell (x,\zeta (y))\le 1$ and then $\zeta(y)\in Q\cap\zeta(Q)$. 
So $Q=\zeta(Q)$ and, consequently, for every $z \in P$, $z\in \mathrm{Min}(Q)$ if and only if $\zeta(z)\in \mathrm{Max}(Q)$. Thus $|\mathrm{Min}(Q)|=|\mathrm{Max}(Q)|$.\\

(6) Suppose $P$ is of $\zeta$-width $1$ and $|Q|>2$. As, by (3), $Q=\mathrm{Min}(Q)\cup \mathrm{Max}(Q)$ and by (5), $|\mathrm{Min}(Q)|=|\mathrm{Max}(Q)|$, we have that $|\mathrm{Min}(Q)|\geq 2$. That $Q=\zeta (Q)$ follows from  (5). Let $x, y$ be distinct elements in $\mathrm{Min}(Q)$. Then $\ell (x,y)\ge 2$ and, consequently $\ell (x,\zeta (y))\le 1$. 
Since $|Q|>2$ and $x\in \mathrm{Min}(Q)$ and $\zeta (y)\in \mathrm{Max}(Q)$, we have that $x\neq \zeta(y)$ and, consequently, $\ell (x,\zeta (y))=1$ and hence $x<\zeta(y)$.
\end{proof}

 \begin{Theorem}\label{D2}  Let $L$ be a non-trivial $pm$-algebra and $P$ be its dual $pm$-space. Then $L\in \mathbf{M}_1$ if and only if, for every order component $Q\subseteq P$, $Q\cup \zeta (Q)$ is a copy of one of $Q_0$, $Q_1$,  $Q_2$, $Q_3$, $Q_4$, $Q_5$, or is of type $Q_6(I,S)$ for some Stone space $(S;\tau)$ and $I\subseteq S$ a subset of the isolated points of $S$ with $3\leq |S|$  {\rm (}as diagrammed in {\rm Figure $1$)}.
\end{Theorem}

\begin{proof}
Applying Theorem \ref{C1} and Corollary \ref{C4B}, we know that  $L\in \mathbf{M}_1$ if and only if $P$ has height at most $1$ and, for any order component $Q$ of $P$ and any $x,y\in Q$, we have $\ell_{\zeta} (x,y)\leq 1$.\\

Suppose that for every order component $Q\subseteq P$, $Q\cup \zeta (Q)$ is a copy of one of $Q_0$, $Q_1$,  $Q_2$, $Q_3$, $Q_4$, $Q_5$, or is of type $Q_6(I,S)$ with $3\leq |S|$. It is obvious that $P$ has height at most $1$. It is also clear that if $x,y\in Q_i$, $0\leq i\leq 5$, then $\ell_{\zeta} (x,y)\leq 1$. Consider $Q_6(I,S)$ with $3\leq |S|$. Suppose $x,y\in S$ such that $x\neq y$. Then $x<\zeta (y)$ and so $\ell_{\zeta} (x,y)=1$. As $\ell_{\zeta}(x,y)=\ell_{\zeta}(\zeta(x),\zeta (y))=\ell_{\zeta}(x,\zeta (y))=\ell_{\zeta}(\zeta(x),y)$, we have that, for any $x,y\in Q_6(I,S)$, $\ell_{\zeta} (x,y)\leq 1$. So $L\in \mathbf{M}_1$.\\

Conversely, suppose $L\in \mathbf{M}_1$ and let $Q$ be an order component of $P$. If $|Q| = 1$, then  $Q\cup \zeta (Q)$ is a copy of $Q_0$ or $Q_1$.\\

Suppose $|Q| = 2$. By (4) of the previous lemma, we know that $x\neq\zeta(x)$, for any $x\in Q$. So $Q\cup \zeta (Q)$ is a copy of $Q_2$ if $Q=\zeta(Q)$, whilst $Q\cup \zeta (Q)$ is a copy of $Q_3$ if $Q\cap \zeta(Q)=\emptyset$.\\

Let $|Q| > 2$.  Observe that, by the previous lemma, $Q$ is the disjoint union of $\mathrm{Min}(Q)$ and $\mathrm{Max}(Q)$, $Q=\zeta(Q)$, $|\mathrm{Max}(Q)|=|\mathrm{Min}(Q)|\geq 2$ and, for distinct $x,y$ in $\mathrm{Min}(Q)$, $x<\zeta(y)$, each of which will hold for the remainder of the proof.\\

Suppose initially that $|\mathrm{Min}(Q)| = 2$ and, say, $\mathrm{Min}(Q)=\{x,y\}$.  Then $\mathrm{Max}(Q)=\{\zeta(x),\zeta(y)\}$. Were it the case that both $x\not\leq \zeta (x)$ and $y\not\leq \zeta (y)$, then $Q$ would fail to be an order component.  
Thus, either $x\leq \zeta (x)$ and $y\leq \zeta (y)$, and $Q\cup \zeta (Q)=Q$ is a copy of $Q_5$, or else, say, $x\not\leq \zeta (x)$ and $y\leq \zeta (y)$ and $Q\cup \zeta (Q)=Q$ is a copy of $Q_4$. \\

 Suppose next that $|\mathrm{Min}(Q)| \ge 3$. We have that $\mathrm{Max}(Q)=\zeta (\mathrm{Min}(Q))$. 
 As $L\in \mathbf{M}_1$, applying Theorem \ref{X4}, we know that $Q$ is a closed set. Since $\zeta(Q)=Q$ and it is obvious that $\mathrm{Min}(x)\subseteq Q$, whenever $x\in Q$, we have that $(Q;\tau_{\restrict Q}, \le_{\restrict Q},\zeta_{\restrict Q})$ is a $pm$-subspace of $P$ for which $\mathrm{Min}(Q)$ and $\mathrm{Max}(Q)$ are both closed and then clopen in $Q$, as $Q$ is the disjoint union of these sets. We have that $\mathrm{Min}(Q)$ endowed with the induced topology is a Stone space. As $\zeta$ is a homeomorphism, $\mathrm{Max}(Q)=\zeta( \mathrm{Min}(Q))$ is a homeomorphic copy of $\mathrm{Min}(Q)$. 
Consider the $pm$-subspace $(Q;\tau_{\restrict Q}, \le_{\restrict Q},\zeta_{\restrict Q})$.
Let $x\in\mathrm{Min}(Q)$.  As observed above $x<\zeta(z)$ for any $z\in \mathrm{Min}(Q)\setminus \{x\}$. Suppose $x$ is an accumulation point of  $\mathrm{Min}(Q)$. If $x\not\leq \zeta(x)$, then there exists a clopen decreasing set $X$ containing $\zeta(x)$ but not $x$. Since $x$ is an accumulation point of  $\mathrm{Min}(Q)$, 
$\zeta(x)$ is an accumulation point of  $\mathrm{Max}(Q)$ and there exists $y\in\mathrm{Max}(Q)\cap(X\setminus\{\zeta(x)\})$. As observed above $x<\zeta(\zeta(y))=y$ and so $x\in X$, since $y\in X$ and $X$ is decreasing, a contradiction.
 Now consider the set $I=\{x\in \mathrm{Min}(Q)\colon x\not< \zeta(x)\}$. The elements of this set are isolated points of $\mathrm{Min}(Q)$. Thus $Q$ is a $pm$-subspace of type $Q_6(I,S)$ where $S$ denotes the Stone space $\mathrm{Min}(Q)$ and $I=\{x\in \mathrm{Min}(Q)\colon x\not< \zeta(x)\}\subseteq \mathrm{Min}(Q)$ is a set of isolated points of $\mathrm{Min}(Q)$.\end{proof}
 
 The next corrollary is an immediate consequence of Theorems \ref{X5} and \ref{D2}.

\begin{Corollary}\label{D2A} The simple algebras in $\mathbf{M}_1$ are the ones whose dual space is a copy of   $Q_0$, $Q_1$,  $Q_2$, $Q_3$, $Q_4$, $Q_5$,  or is of type $Q_6(I,S)$ for some Stone space $(S;\tau)$ and $I\subseteq S$ a subset of the isolated points of $S$ with $3\leq |S|$  {\rm (}as diagrammed in {\rm Figure $1$)}.\hfill{$\Box$}
\end{Corollary}

Note that $Q_0$,   $Q_2$,   $Q_5$,  and $Q_6(\emptyset,S)$ represent Kleene algebras, but no others do.\\

As every subdirectly irreducible algebra in $\mathbf{M}_n$, $n<\omega$, is simple, applying Corollaries \ref{D1A} and \ref{D2A}, we have that $\mathbf{M}_0$ is generated by the pseudocomplemented de Morgan algebras  $E(Q_0)$, $E(Q_1)$, and $E(Q_2)$, whilst  $\mathbf{M}_1$ is generated by the pseudocomplemented de Morgan algebras  $E(Q_0)$, $E(Q_1)$,  $E(Q_2)$, $E(Q_3)$, $E(Q_4)$, $E(Q_5)$,  and $E(Q_6(I,S))$ where $(S;\tau )$ is a Stone space, $3\leq |S|$, and $I\subseteq S$ is a subset of the set of isolated points of $S$.\\

Recall that a Stone space is finite if and only if it is discrete. Consequently, all the elements of a finite Stone space are isolated points. It is clear that, up to $pm$-isomorphism, for any $n$, $3\leq n<\omega$, and any $m$, $0\leq m\leq n$, there is exactly one $Q_6(I, S)$ such that $|S|=n$ and $|I|=m$. Consider such an $S$ and such an $I$, we denote by $Q_6(m,n)$ the $pm$-space $Q_6(I,S)$.\\

  In fact, as the following {theorem} shows,  $\mathbf{M}_1$ is generated by the pseudocomplemented de Morgan algebras represented by $Q_0$, $Q_1$,  $Q_2$, $Q_3$, $Q_4$, $Q_5$, and $Q_6(m,n)$ where $3\le n<\omega$ and $0\le m\le n$.\\

\begin{Theorem}\label{D3}
The variety $\mathbf{M}_1$ is locally finite.
\end{Theorem}
\begin{proof}
Since the pseudocomplemented de Morgan algebras $E(Q_0)$, $E(Q_1)$,  $E(Q_2)$, $E(Q_3)$, $E(Q_4)$, $E(Q_5)$,  and $E(Q_6(I,S))$ where $(S;\tau )$ is a Stone space, $3\leq |S|$, and $I\subseteq S$ is a subset of the set of its isolated points generate $\mathbf{M}_1$, it is sufficient to show that for each $N<\omega$, there exists $N^\prime <\omega$, depending on $N$, such that, for any set of $N$ elements in any one of the pseudocomplemented de Morgan algebras $E(Q_0)$, $E(Q_1)$,  $E(Q_2)$, $E(Q_3)$, $E(Q_4)$, $E(Q_5)$,  and $E(Q_6(I,S))$ where $(S;\tau )$ is a Stone space, $|S|\geq 3$, and $I$ is a subset of the set of its isolated points, generates a subalgebra whose cardinality does not exceed  $N^\prime $ (see Mal'cev \cite[VI.14 Theorem 3]{Ma70}).  Since $\{E(Q_i)\colon 0\leq i\leq 5\}$ is a finite set of finite algebras, it is only required to show that, for each $N<\omega$, there exists $N^\prime <\omega$ such that, for any set of $N$ elements in any one of the pseudocomplemented de Morgan algebras $E(Q_6(I,S))$  generates a subalgebra with at most  $N^\prime $ elements.\\

Consider $E(Q_6(I,S))$. Recall that whenever $z$ is an isolated point of the Stone space $S$ or of the Stone space $\zeta(S)$, $z$ is also an isolated point of $Q_6(I,S)$ and, consequently, $\{z\}$ is clopen in $Q_6(I,S)$. Let $x\in I$. Then $x$ is an isolated point of $S$, $\zeta(x)$ is an isolated point of $\zeta(S)$ and we have that
\[ (\zeta(x)] = \{ \zeta(x)\} \cup (S\setminus \{x\}), \]
 is a clopen decreasing set of $Q_6(I,S)$.\\

\noindent {\bf Claim.} Let $\mathcal{F}$ be a Boolean subalgebra of the Boolean algebra of clopen sets of $S$, let  $\mathcal{F}_I=\{\{x\}\in \mathcal{F}\colon x\in I\}$ and
  \[K(\mathcal{F}) :=\mathcal{F}\cup \{\zeta(X)\cup S \colon X\in \mathcal{F}\}\cup \{(\zeta(x)]\colon \{x\}\in \mathcal{F}_I\}.\]  Then $K(\mathcal{F})$ is a subalgebra of $E(Q_6(I,S))$.\\

First notice that all the elements of $K(\mathcal{F})$ are clopen decreasing sets of $Q_6(I,S)$ and, since $\emptyset, S\in\mathcal{F}$, we have that $\emptyset\in K(\mathcal{F})$ and $S\cup \zeta(S)\in K(\mathcal{F})$.\\

We begin by showing that $K(\mathcal{F})$ is closed under unions and intersections. Consider the various cases, starting with those that involve the first two types of clopen decreasing sets for which there are six cases.\\

 Let $X,Y\in \mathcal{F}$. Obviously $X\cup Y, X\cap Y\in \mathcal{F}$. Further, we have that $X\cup (\zeta(Y)\cup S)=\zeta(Y)\cup S$ and $X\cap (\zeta(Y)\cup S)=X$, whilst $(\zeta(X)\cup S)\cup (\zeta(Y)\cup S)=\zeta(X\cup Y)\cup S$ and $(\zeta(X)\cup S)\cap (\zeta(Y)\cup S)=\zeta(X\cap Y)\cup S$. It remains to consider the cases involving the third type of clopen decreasing set for which there are another six cases. Let $\{x\}\in \mathcal{F}_I$. Then $\{x\}\in \mathcal{F}$, $x\in I$ and $S\setminus\{x\}\in\mathcal{F}$.  We have $X\cap (\zeta(x)]=X\cap (\{\zeta(x)\}\cup (S\setminus \{x\}))=X\cap (S\setminus \{x\})$ and $(\zeta(X)\cup S)\cup (\zeta(x)]=\zeta(X\cup \{x\})\cup S$. Four cases remain. First consider $X\cup (\zeta(x)]$ and $(\zeta(X)\cup S)\cap (\zeta(x)]$. If $x\notin X$, then
$X\cup (\zeta(x)]=X\cup\{\zeta(x)\}\cup (S\setminus \{x\})=\{\zeta(x)\}\cup (S\setminus \{x\})=(\zeta(x)]$ and $(\zeta(X)\cup S)\cap (\zeta(x)]=S\setminus \{x\}$, whilst, if $x\in X$, $X\cup (\zeta(x)]=\{\zeta(x)\}\cup S=\zeta(\{x\})\cup S$ and $(\zeta(X)\cup S)\cap (\zeta(x)]=(\zeta(x)]$. Finally, two cases remain, each involving the third type of clopen decreasing sets. Let $\{y\}\in \mathcal{F}_I$. Then $\{y\}\in \mathcal{F}$ and $y\in I$. Suppose $x\neq y$. Notice that $\{x,y\}\in \mathcal{F}$. We have $(\zeta(x)]\cup (\zeta(y)]= \zeta(\{x,y\})\cup S$  and $(\zeta(x)]\cap (\zeta(y)]=S\setminus\{x,y\}$. \\

Now we show that $K(\mathcal{F})$ is closed under pseudocomplementation. Let $X\in \mathcal{F}$. If $X=\emptyset$, then $X^{\ast}=S\cup\zeta(S)$. If $X=\{x\}$ with $x\in I$, then $\{x\}\in\mathcal{F}_I$ and $X^{\ast}=(S\cup\zeta(S))\setminus [x)=\{\zeta(x)\}\cup (S\setminus \{x\})=(\zeta (x)]$. For the remaining cases, $X^{\ast}=(S\cup\zeta(S))\setminus [X)=(S\cup\zeta(S))\setminus (X\cup \zeta(S))=S\setminus X$. Further, we have $(\zeta(X)\cup S)^{\ast}= \emptyset$. It remains to show that $K(\mathcal{F})$ is closed for the third type of clopen decreasing set. Let $\{x\}\in \mathcal{F}_I$. Then $\{x\}\in \mathcal{F}$ and $x\in I$ and we have $(\zeta(x)]^{\ast}=(S\cup \zeta(S))\setminus [\{\zeta(x)\}\cup (S\setminus\{x\}))=\{x\}$.\\

Finally we show that $K(\mathcal{F})$ is closed for the de Morgan operation. Let $X\in \mathcal{F}$. We have $X^{\prime}=(S\cup \zeta(S))\setminus \zeta (X)=\zeta (S\setminus X)\cup S$ and $(\zeta(X)\cup S)^{\prime}=(S\cup \zeta(S))\setminus (X\cup \zeta (S))=S\setminus X$. If $\{x\}\in \mathcal{F}_I$, we have $(\zeta(x)]^{\prime}=(S\cup \zeta(S))\setminus (\{x\}\cup( \zeta(S)\setminus \{\zeta(x)\}))=\{\zeta(x)\}\cup (S\setminus\{x\})=(\zeta(x)]$, thereby completing the justification of the claim.\\

Let $N$ be a positive integer. Let $\{ X_i\colon 0\leq i<N\}$ be a set of clopen decreasing sets of $Q_6(I,S)$. For each $i$, $0\leq i<N$, $X_i=(X_i\cap S)\cup (X_i\cap \zeta(S))= (X_i\cap S)\cup \zeta(\zeta (X_i)\cap S)$ and $X_i\cap S$ and $\zeta (X_i)\cap S$ are clopen sets of $S$. Consider $T=\{X_i\cap S\colon 0\leq i<N\}\cup \{\zeta (X_i)\cap S\colon 0\leq i<N\}$. This set has, at most, $N_1 = 2N$ elements. Let $\mathcal{B}$ be the Boolean algebra of clopen sets of $S$ and $\mathcal{F}$ its Boolean subalgebra generated by $T$. Since a free Boolean algebra with $N_1$ generators has $2^{(2^{N_1})}$ elements,
$\mathcal{F}$ has at most $N_2=2^{(2^{N_1})}$ elements. By the above claim, $K(\mathcal{F})$ is a subalgebra of $E(Q_6(I,S))$. It is obvious that $K(\mathcal{F})$ has, at most, $3N_2$ elements. It remains to show that, for every $i$, $0\leq i<N$, $X_i\in K(\mathcal{F})$. We know that $X_i=(X_i\cap S)\cup \zeta(\zeta (X_i)\cap S)$. For each $i$, $ X_i\cap S\in T\subseteq \mathcal{F}$ and $\zeta(X_i)\cap S\in T\subseteq \mathcal{F}$. If $\zeta(X_i)\cap S=\emptyset$, then $X_i=X_i\cap S\in \mathcal{F}\subseteq K(\mathcal{F})$. Suppose $|\zeta(X_i)\cap S|\geq 2$ or $\zeta(X_i)\cap S=\{x\}$ with $x\notin I$. Then $|\zeta(\zeta(X_i)\cap S)|\geq 2$ or $\zeta(\zeta(X_i)\cap S)=\{\zeta(x)\}$ with $x\notin I$. As $\zeta(\zeta(X_i)\cap S)\subseteq X_i$ and $X_i$ is decreasing, $S\subseteq X_i$ and, consequently, $X_i=S\cup \zeta (\zeta(X_i)\cap S)\in K(\mathcal{F})$. Finally, suppose $\zeta(X_i)\cap S=\{x\}$ with $x\in I$. Then $\zeta(\zeta (X_i)\cap S)=\{\zeta(x)\}$ and $X_i=(X_i\cap S)\cup \{\zeta(x)\}$ and, as $X_i$ is decreasing, $X_i=(X_i\cap S)\cup (\zeta(x)]$ with $\{x\}\in\mathcal{F}$ and $x\in I$. So, $X_i\in K(\mathcal{F})$.\\

 Thus, the subalgebra of $E(Q_6(I,S))$ generated by $\{X_i\colon 0\leq i<N \}$ has at most $N^{\prime}=3N_2 =3\times 2^{(2^{N_1})}= 3\times 2^{(2^{2N})}$ elements.
\end{proof}

For any $i$, $0\leq i\leq 5$,  and for any pair $(m,n)$ with $3\leq n<\omega$ and  $0\leq m\leq n$, let 
\[ L_i=E(Q_i)\;\;\; \mbox{ and }\;\;\; L_6(m,n)=E(Q_6(m,n)).\]

 In view of Theorem \ref{D3}, the following corollary is immediate.
 
 \begin{Corollary} \label{D3A}
 The variety  $\mathbf{M}_1$ is generated by its finite simple algebras, which are, up to isomorphism, $L_i$, $0\leq i\leq 5$, and $L_6(m,n)$, $3\leq n<\omega$,  $0\leq m\leq n$.\hfill{$\Box$}
\end{Corollary} 

To show that  the free algebra in $\mathbf{K}_2$ on one generator, $F_{\mathbf{K}_2}(1)$, is infinite we are going to consider the sequence of algebras of $\mathbf{K}_2$ presented in the following example.

\begin{Example}\label{F1}  {\rm For $5\leq n<\omega $, let $S_n = \{ x_i\colon 0\leq i<n \}$ and $T_n = \{ y_i\colon 0\leq i<n \}$ be disjoint $n$-element sets. 
Define on $S_n\cup T_n$ the partial order $\leq$  induced by, for $0\leq i,j<n$, 
\[x_i<y_j \;\;\mbox{ iff }\;\; i\notin \{j-1, j+1\}.\] 
 Let $\zeta\colon S_n\cup T_n \to S_n\cup T_n$ be defined by $\zeta(x_i)=y_i$ and $\zeta(y_i)=x_i$, for any $0\leq i<n$.  Consider $P_n = ( S_n\cup T_n;\leq ,\zeta )$ which is a $pm$-space and its dual $pm$-algebra $K_n=E(P_n)$. \\
 
 As ${\rm Min}(P_n)=S_n$ and ${\rm Max}(P_n)=T_n$, we have that $P_n={\rm Min}(P_n)\cup {\rm Max}(P_n)$ and then $K_n$ is a regular $pm$-algebra, by Theorem \ref{C1}. Moreover,
 since $\zeta(y_i)=x_i<y_i=\zeta (x_i)$ for every $0\leq i<n$, $K_n$ is a regular pseudocomplemented Kleene algebra.\\
  
 The algebra $K_n$ is of range $2$, that is $K_n\in\mathbf{K}_2$. In fact, since  $\ell_{\zeta}(x,y)=\ell_{\zeta}(\zeta(x),\zeta(y))=\ell_{\zeta}(x,\zeta(y))=\ell_{\zeta}(\zeta(x),y)$, for any $x,y\in P_n$, it suffices to show that, for $0\leq i,j<n$, $\ell_{\zeta}(y_i,y_j)\leq 2$. Let $0\leq i,j<n$, choose $0\leq k<n$ such $k\notin \{i-1, i+1, j-1, j+1\}$. Then $x_k<y_i ,y_j$ and $\ell(y_i,y_j)\leq 2$. So $\ell_{\zeta}(y_i,y_j)\leq 2$. Thus $P_n$ is of $\zeta$-width $2$ and then $K_n$ is of range $2$, by Corollary \ref{C4B}.\\
  
 Notice that $K_n\notin\mathbf{K}_1$, since $\ell(y_0,y_1)=2$ and $\ell(y_0,\zeta(y_1))=\ell(y_0, x_1)=3$, so $\ell_{\zeta}(y_0,y_1)=2$.  }
\end{Example}

\begin{Theorem}\label{D6}   The variety $\mathbf{K}_2$ is not locally finite, in fact the free algebra on one generator  $F_{\mathbf{K}_2}(1)$ is infinite.  
\end{Theorem}

\begin{proof}  
We show that, for each $5\leq n<\omega$, there exists a subalgebra $A_n$ of $K_n$ that is generated by one element and has cardinality, at least, $n$. \\

Let $5\leq n<\omega$.
First notice that all the sets $\{x_i\}$ are decreasing. We have $[x_0)=\{x_0\}\cup (T_n\setminus \{y_1\})$, $\{x_0\}^{\ast}=P_n\setminus [x_0)=\{y_1\}\cup (S_n\setminus \{x_0\})$ and $\{x_0\}^{\ast \prime}=\zeta (P_n\setminus \{x_0\}^{\ast})= \zeta (P_n\setminus (P_n\setminus [x_0)))= \zeta([x_0))=\{y_0\}\cup (S_n\setminus \{x_1\})$. Moreover, for $1\leq i< n-1$, we have $[x_i)=\{x_i\}\cup (T_n\setminus \{y_{i-1}, y_{i+1}\})$, $\{x_i\}^{\ast}=\{y_{i-1}, y_{i+1}\}\cup (S_n\setminus\{x_i\})$ and $\{x_i\}^{\ast \prime}=\zeta([x_i))=\{y_i\}\cup (S_n\setminus \{x_{i-1}, x_{i+1}\})$.\\ 

Consider $X = \{x_0\}$ and let $\langle X \rangle$ denote the subalgebra of $K_n$ generated by $X$. We will  show that, for every $0\leq i<n$, $\{x_i\}\in \langle X \rangle$. 
Since $X^{\ast\prime} = \{x_0\}^{\ast\prime}= \{y_0\}\cup (S_n\setminus \{x_1\})$, we have that $X^{\ast\prime\ast} =P_n\setminus [X^{\ast\prime})= \{x_1\}\in \langle X \rangle$.
Proceeding inductively,  let $1\leq i< n-1$ and suppose that $ \{x_k\}\in \langle X \rangle$, for any $0\leq k\leq i$. Since
$\{x_i\}^{\ast\prime} =\{y_i\}\cup (S_n\setminus \{x_{i-1}, x_{i+1}\})$, we have that
$\{x_i\}^{\ast\prime\ast} =P_n\setminus [\{x_i\}^{\ast\prime})=\{x_{i-1}, x_{i+1}\}$. As $x_{i-1} \notin \{x_{i-1}\}^{\ast}$ but $x_{i+1} \in \{x_{i-1}\}^{\ast}$, we have  $\{x_i\}^{\ast\prime\ast}\cap \{x_{i-1}\}^{\ast}=\{x_{i+1}\}$ and, consequently, $\{x_{i+1}\}\in \langle X \rangle$, since both $\{x_i\}^{\ast\prime\ast}, \{x_{i-1}\}^{\ast}\in \langle X \rangle$.\\

As all the regular pseudocomplemented Kleene algebras $A_n$ are homomorphic images of $F_{\mathbf{K}_2}(1)$, this algebra must be infinite. 
\end{proof}

\begin{Corollary}\label{D7}
For $2\leq n<\omega$, the varieties ${\mathbf{K}_n}$ and ${\mathbf{M}_n}$ are not locally finite. 
\end{Corollary}

\begin{proof}
It is immediate from Theorem \ref{D6} that  $F_{\mathbf{V}}(1)$ is infinite, where $\mathbf{V}$ is any of the mentioned varieties.  
\end{proof}

\section{Lattices of subvarieties of $\mathbf{K}_1$ and $\mathbf{M}_1$} \label{SE}

As observed in Section \ref{SA} the following theorem follows from 
\cite{Sa87a}.

\begin{Theorem} \label{E1}{\rm (\cite{Sa87a})}

{\rm (1)} $L_V(\mathbf{K}_0)$ is isomorphic to a $3$-element chain,

{\rm (2)} $L_V(\mathbf{M}_0)$ is isomorphic to the $5$-element lattice ${\mathbf 1}\oplus ({\mathbf 2}\times{\mathbf 2})$.
\end{Theorem}

Now we focus on $L_V(\mathbf{K}_1)$ and on $L_V(\mathbf{M}_1)$.\\

As  the variety of pseudocomplemented de Morgan algebras is congruence-distributive, it is well known  (cf. \cite{Jo67}) that its lattice of subvarieties is distributive. \\

Let $\mathcal{K}$ be a class of algebras of the same similarity type. As usual, we denote by $\mathbf{H}(\mathcal{K})$, $\mathbf{S}(\mathcal{K})$, $\mathbf{I}(\mathcal{K})$, and $\mathbf{V}(\mathcal{K})$ respectively, the class of all homomorphic images of members of $\mathcal{K}$, the class of all subalgebras of members of $\mathcal{K}$,  the class of all isomorphic images of members of $\mathcal{K}$, and the variety generated by $\mathcal{K}$. If $\mathcal{K}$ has finitely many members, say $\mathcal{K}=\{K_0, \ldots, K_n\}$ with $n<\omega$, the mentioned classes  will be denoted by $\mathbf{H}(K_0, \ldots, K_n)$, $\mathbf{S}(K_0, \ldots, K_n)$, $\mathbf{I}(K_0, \ldots, K_n)$ and $\mathbf{V}(K_0, \ldots, K_n)$, respectively.\\

\begin{Theorem}\label{E0} Let $ n<\omega$ and let $\mathcal{K}$ be a class of simple algebras of $\mathbf{M}_n$ such that $\mathbf{V}(\mathcal{K})$ is locally finite. Then every finite simple algebra in  $\mathbf{V}(\mathcal{K})$ belongs to $\mathbf{IS}(\mathcal{K})$.
\end{Theorem}
\begin{proof}
We have already observed that any subalgebra of a simple algebra in $\mathbf{M}_n$ is also simple. So $\mathbf{HS}(\mathcal{K})=\mathbf{IS}(\mathcal{K})$. Now the conclusion follows immediately from \cite[Theorem 4.104]{ McMcTa87}.
\end{proof}

The set $Si_F=\{L_i\colon 0\leq i\leq 5\}\cup\{L_6(m,n)\colon 3\leq n<\omega \mbox { and } 0\leq m\leq n\}$ consists of precisely one algebra from each of the isomorphism classes of the finite simple algebras in $\mathbf{M}_1$. Since, by Theorem \ref{D3}, $\mathbf{M}_1$ is locally finite, any of its subvarieties is generated by its finite simple members. So
any subvariety of $\mathbf{M}_1$ is generated by some subset $\mathcal{K}$ of  $Si_F$. Moreover, in view of the last theorem, every finite simple   member of $\mathbf{V}(\mathcal{K})$ lies in $\mathbf{IS}(\mathcal{K})$. \\

As observed in Section \ref{SB}, if $P$ and $Q$ are the dual spaces of the $pm$-algebras $L$ and $K$, respectively, then $L$ is isomorphic to a subalgebra of $K$ if and only if there exists a surjective $pm$-morphism from $Q$ onto $P$.  If $L$ and $K$ are finite, the topology in $P$ and $Q$ is the discrete one and plays no role. Notice that for any  $pm$-morphism $\varphi\colon Q\to P$, we have $\varphi({\rm Min}(Q))\subseteq {\rm Min}(P)$ and $\varphi({\rm Max}(Q))\subseteq {\rm Max}(P)$ and, since $\varphi$ is order-preserving, if each $Q$ and $P$ is the disjoint union of the set of its minimal elements and the set of its maximal ones, we have $\varphi(x)<\varphi(y)$ whenever $x<y$. Moreover, for any surjective $pm$-morphism $\varphi\colon Q\to P$, we have $\varphi({\rm Min}(Q))={\rm Min}(P)$ and $\varphi({\rm Max}(Q))={\rm Max}(P)$.\\

Ultimately to describe $L_V(\mathbf{M}_1)$ it will be necessary to determine precisely when $L_6(p,q)$ is a subalgebra of $L_6(m,n)$ for all $m, n, p, q$ with $3\leq n<\omega$, $0\leq m\leq n$, $3\leq q<\omega$ and $0\leq p\leq q$. The next two lemmas \ref{E3A} and \ref{E3B} pave the way for this, namely Theorem \ref{E3}.

\begin{Lemma} \label{E3A}
Let $Q_6(I,S)$ and $Q_6(J,T)$ be finite. A map $\varphi\colon Q_6(I,S)\to Q_6(J,T)$ is a surjective $pm$-morphism if and only if the following conditions are satisfied\\
{\rm (1)} $\varphi(S)=T$ and $\varphi(\zeta(x))=\zeta(\varphi(x))$, for any $x\in S$.\\
{\rm (2)} $\varphi^{-1}(J)\subseteq I$.\\
{\rm (3)}  $\varphi(x)\neq\varphi(y)$, for distinct $x,y\in \varphi^{-1}(J)$.\\
{\rm (4)} If $x\in I\setminus \varphi^{-1}(J)$, then there exists $u\in S\setminus \varphi^{-1}(J)$ such that $u\neq x$ and $\varphi(u)=\varphi(x)$. 
\end{Lemma}
\begin{proof}
We have that ${\rm Min}(Q_6(I,S))=S$, ${\rm Max}(Q_6(I,S))=\zeta(S)$, whilst ${\rm Min}(Q_6(J,T))=T$ and ${\rm Max}(Q_6(J,T))=\zeta(T)$. Moreover, $S$ and $\zeta(S)$ are disjoint and so are $T$ and $\zeta(T)$.\\

Suppose $\varphi\colon Q_6(I,S)\to Q_6(J,T)$ is a surjective $pm$-morphism. \\

(1) It is obvious since $\varphi(S)=\varphi({\rm Min}(Q_6(I,S)))={\rm Min}(Q_6(J,T))=T$ and $\varphi \circ \zeta=\zeta\circ\varphi$.\\

(2) Let $x\in \varphi^{-1}(J)$. Then $\varphi(x)\in J$ and, consequently, $x\in S$ and  $\varphi(x)\not <\zeta(\varphi(x))=\varphi(\zeta(x))$.  If $x\notin I$, then $x<\zeta(x)$ and,  consequently, $\varphi(x)<\varphi(\zeta(x))$, a contradiction. Thus, $x\in I$.\\

(3) Let $x,y\in \varphi^{-1}(J)$ be such that $x\neq y$. Then $x<\zeta (y)$ and, consequently, $\varphi(x)<\varphi(\zeta (y))=\zeta(\varphi (y))$. If $\varphi (x)=\varphi(y)$, we would have 
$\varphi(x)<\zeta(\varphi (x))$, a contradiction, since $\varphi(x)\in J$. Thus, $\varphi (x)\neq\varphi(y)$.\\

(4) Suppose that $x\in I\setminus \varphi^{-1}(J)$. That is, $x\in I$ and $\varphi(x)\notin J$. Then $x\not < \zeta(x)$ and $\varphi(x)<\zeta(\varphi(x))=\varphi(\zeta(x))$. So,  $\varphi(x)\in {\rm Min}(\varphi(\zeta(x)))=\varphi({\rm Min}(\zeta(x)))$ and, consequently, there exists $u\in {\rm Min}(\zeta(x))$ such that $\varphi(x)=\varphi(u)$. It is obvious that $u\in S$ and, as $\varphi(u)=\varphi(x)\notin J$, we have that $u\in S\setminus \varphi^{-1}(J)$. From $u<\zeta(x)$ and $x\not < \zeta(x)$, it follows that $u\neq x$.\\

Conversely, suppose that $\varphi$ satisfies (1) - (4). By (1), $\varphi$ is onto and $\varphi\circ \zeta=\zeta\circ \varphi$. To prove that $\varphi$ is order-preserving, consider $x,y\in S$ such that $x<\zeta (y)$. If $\varphi(x)\not <\varphi(\zeta (y))=\zeta (\varphi(y))$, then $\varphi(x)=\varphi(y)\in J$ and, applying (2) and (3), $x=y\in \varphi^{-1}(J)\subseteq I$, a contradiction, since $x<\zeta (y)$.  Finally, let $x\in S$. As $S={\rm Min}(Q_6(I,S))$, $T={\rm Min}(Q_6(J,T))$ and, by (1), $\varphi(S)=T$, we have that  ${\rm Min}(\varphi(x))\subseteq \varphi({\rm Min} (x))$ trivially. To prove that ${\rm Min}(\varphi(\zeta(x)))\subseteq \varphi({\rm Min} (\zeta(x)))$, take $z\in {\rm Min}(\varphi(\zeta(x)))$. Then $z\in T=\varphi(S)$ and $z<\varphi(\zeta(x))=\zeta (\varphi(x))$. Let $y\in S$ be such that $z=\varphi(y)$ and suppose that $y\not < \zeta(x)$. Then $y=x\in I$ and $\varphi(x)=z<\zeta (\varphi(x))$. So $x\notin \varphi^{-1}(J)$. Applying (4), there exists $u\in S\setminus \varphi^{-1}(J)$ such that $u\neq x$ and $\varphi(u)=\varphi(x)=z$. Since $u\neq x$, we have $u<\zeta(x)$ and then $z\in \varphi({\rm Min} (\zeta(x)))$. So,  ${\rm Min}(\varphi(\zeta(x)))\subseteq \varphi({\rm Min} (\zeta(x)))$.                  Thus, $\varphi$ is a surjective $pm$-morphism.
\end{proof}

It is straightforward to conclude that for $3\leq  k\leq n<\omega$ and $0\leq m\leq k$, there is a surjective $pm$-morphism from $Q_6(m,n)$ onto $Q_6(m,k)$. Simply apply the previous lemma to any $\varphi\colon Q_6(I,S)\to Q_6(J,T)$ where $|I|=|J|$ and $|S\setminus I| \geq |T\setminus J|$ and $\varphi_{\restrict I}\colon I\to J$ is a bijection and $\varphi_{\restrict (S\setminus I)}\colon S\setminus I \to T\setminus J$ is an onto mapping. \\

Now we introduce some notation.
 For each $i$, $0\leq i\leq 5$, and each pair $(m,n)$ such that $3\leq n <\omega$ and $0\leq m\leq n$, let 
\begin{center}
$\mathbf{V}_i=\mathbf{V}(L_i)$ and $\mathbf{V}_{m,n}=\mathbf{V}(L_6(m,n))$.
\end{center}

Also let
 
\begin{center}
 $\mathbf {V}_{\omega ,\omega} = \mathbf{V}( \{L_6(m,n)\colon 3\leq n<\omega \mbox{ and } 0\leq m<n \})$,
\end{center} 
and, for each $ m<\omega$,
\begin{center}
 
$\mathbf{ V}_{m ,\omega} = \mathbf{V}(\{L_6(m,n)\colon 3\leq n<\omega \mbox{ and }  
m < n\})$.
\end{center}

\begin{Theorem} \label{E2}
$L_V({\mathbf{K_1}})$ is isomorphic to an $\omega +1$ chain.
\end{Theorem}
\begin{proof}
Analysing the algebras of $Si_F$, we conclude that only $L_0, L_2, L_5$ and $L_6(0,n)$, for any $3\le n<\omega$, are in $\mathbf{ K}_1$.
As we have just observed, for $3\leq n<m$, there is a surjective $pm$-morphism from $Q_6(0,m)$ onto $Q_6(0,n)$, but, obviously, there is no surjective map from $Q_6(0,n)$ onto $Q_6(0,m)$. So, $L_6(0,n)\in \mathbf{IS}(L_6(0,m))$, but $L_6(0,m)\notin \mathbf{IS}(L_6(0,n))$ and, consequently, ${\mathbf V}_{0,n}< {\mathbf V}_{0,m}$.
It is also clear that  there are surjective $pm$-morphisms from $Q_6(0,3)$ onto $Q_5$, from $Q_5$ onto $Q_2$, and from $Q_2$ onto $Q_0$, but not the other way around.  That is, $L_V({\mathbf{K_1}})$ is isomorphic to the $\omega +1$ chain
\[ {\mathbf T}\covered{\mathbf V}_0\covered{\mathbf V}_2
\covered{\mathbf V}_5
\covered{\mathbf V}_{0,3}
\covered\ldots
\covered{\mathbf V}_{0,n}
\covered\ldots <
{\mathbf{K_1}}, \]
where
$\mathbf{T}$ is the trivial subvariety.
\end{proof}

Observe that, in terms of the preceding notation, ${\mathbf{K_1}} = {\mathbf V}_{0,\omega}$.\\

The lattice $L_V({\mathbf{M_1}})$ is a little more complex. Our first goal is to analyse when $L_6(p,q)\in \mathbf{V}_{m,n}$, that is, when $L_6(p,q)\in \mathbf{IS}(L_6(m,n))$.

\begin{Lemma} \label{E3B}
Let $Q_6(I,S)$ and $Q_6(J,T)$ be finite and $\varphi\colon Q_6(I,S)\to Q_6(J,T)$ be a surjective $pm$-morphism. Let $x, y, u, v$ be distinct elements of $S$ such that $x, y\in I\setminus \varphi^{-1}(J)$, $u, v\in S\setminus \varphi^{-1}(J)$ and $\varphi(u)=\varphi(x)\neq \varphi(y)=\varphi (v)$. If, for any finite set $T_1$ that properly contains $T$, there is no surjective $pm$-morphism from $Q_6(I,S)$ onto $Q_6(J,T_1)$, then $u \in I\setminus \varphi^{-1}(J)$ or $ v\in I\setminus \varphi^{-1}(J)$.
\end{Lemma}
\begin{proof}
Suppose that for any finite set $T_1$ that properly contains $T$, there is no surjective $pm$-morphism from $Q_6(I,S)$ onto $Q_6(J,T_1)$. 
As $\varphi$ is a surjective $pm$-morphism, it satisfies conditions (1)-(4) of Lemma \ref{E3A}. Notice that $\varphi(x), \varphi(y)\notin J$, since $x, y\in I\setminus \varphi^{-1}(J)$. As $\varphi(u)=\varphi(x)$ and $\varphi(v)=\varphi(y)$, we have that $u,v\notin \varphi^{-1}(J)$. Thus it only remains to prove that $u\in I$ or $v\in I$.
 Suppose that $u\notin I$ and $v\notin I$. Choose some completely new element $w\notin T\cup\zeta(T)$. Consider $Q_6(J,T_1)$ where $T_1=T\cup\{w\}$. Define $\psi\colon Q_6(I,S)\to Q_6(J,T_1)$ by, for $z\in S$, 
\begin{center}
$\psi(z)=\begin{cases} \varphi(z) &\text{ if } \varphi(z)\notin\{\varphi(x), \varphi(y)\}\\ \varphi(x) &\text{ if } \varphi(z)\in\{\varphi(x), \varphi(y)\} \text{ and } z\notin\{u,v\}\\
\varphi(y) &\text{ if } z=u\\ w &\text{ if } z=v, \end{cases}$\\

\hspace{-9.5cm} and \\
\hspace{-5.5cm}  $\psi(\zeta(z))=\zeta (\psi(z))$.
\end{center}

We will prove that $\psi$ satisfies (1)-(4) of Lemma \ref{E3A}, obtaining a contradiction.\\

(1) We have that $\psi(v)=w$, $\psi(u)=\varphi(y)$ and $\psi(x)=\varphi(x)$. Let $b\in T_1\setminus \{w,\varphi(y), \varphi(x)\}$. Then $b\in T\setminus \{\varphi(y), \varphi(x)\}$ and, since $\varphi(S)=T$, there exists $z\in S$ such that $\varphi(z)=b\notin \{\varphi(y), \varphi(x)\}$ and we have $\psi(z)=\varphi(z)=b$. So $\psi(S)=T_1$. By definition, $\psi(\zeta(z))=\zeta (\psi(z))$ for any $z\in S$.\\
 
(2) Although $\psi^{-1}(J)\subseteq\varphi^{-1}(J)$ is sufficient to conclude (2), we will show that $\psi^{-1}(J)=\varphi^{-1}(J)$ and that $\psi(z)=\varphi(z)$, for $z\in \psi^{-1}(J)$, which will be used to prove (3) and (4). If $z\in \varphi^{-1}(J)$, then $\varphi(z)\in J$ and so $\varphi(z)\notin\{\varphi(x),\varphi(y)\}$. Thus, $\psi(z)=\varphi(z)\in J$ and $z\in\psi^{-1}(J)$.
 Conversely, let $z\in \psi^{-1}(J)$. Then $\psi(z)\in J$. If $\varphi(z)\in  \{\varphi(x), \varphi(y)\}$, we would have $\psi(z)\in \{w, \varphi(y), \varphi(x)\}$, a contradiction, since none of the elements of this set belongs to $J$. So  $\varphi(z)\notin  \{\varphi(x), \varphi(y)\}$ and, consequently, $\varphi(z)=\psi(z)\in J$. Thus $\psi^{-1}(J)= \varphi^{-1}(J)$. Consequently, $\psi^{-1}(J)=\varphi^{-1}(J)\subseteq I$.\\
 
(3)  Let $z_1, z_2\in \psi^{-1}(J)$ be distinct. We just proved in (2) that $\psi^{-1}(J)=\varphi^{-1}(J)$.   Also, from the proof of (2), we know that $\psi(z_1)=\varphi(z_1)$ and $\psi(z_2)=\varphi(z_2)$. As, by (3) of Lemma \ref{E3A}, $\varphi(z_1)\neq\varphi(z_2)$, we have $\psi(z_1)\neq\psi(z_2)$.\\

(4) Suppose that $z\in I\setminus \psi^{-1}(J)$. As $u\notin I$ and $v\notin I$ by hypothesis, we have $z\notin\{u, v\}$. We consider two cases.  First, suppose $\varphi(z)\in \{\varphi(x), \varphi (y)\}$, then 
$\psi(z)=\varphi(x)=\psi(x)=\psi(y)$ by definition of $\psi$, and, since $x\neq y$, we have $z\neq x$ or $z\neq y$. Both $x,y\in I\setminus \varphi^{-1}(J)=I\setminus\psi^{-1}(J)$.  Now suppose $\varphi(z)\notin \{\varphi(x), \varphi (y)\}$. Then $\psi(z)=\varphi(z)$ by definition of $\psi$.  By the proof of (2), we know that $\psi^{-1}(J)=\varphi^{-1}(J)$ and, since $z\in I\setminus \psi^{-1}(J)$, we have $z\in I\setminus \varphi^{-1}(J)$. As $\varphi$ is a surjective $pm$-morphism, applying (4) of Lemma \ref{E3A}, there exists $z_0\in S\setminus \varphi^{-1}(J)=S\setminus \psi^{-1}(J)$ such that $z_0\neq z$ and $\varphi(z_0)=\varphi(z)$.  So $\varphi(z_0)\notin \{\varphi(x), \varphi (y)\}$ and we have $\psi(z_0)=\varphi(z_0)=\varphi(z)=\psi(z)$.
\end{proof}

Recall that, for a real number $r$, the largest integer that does not exceed $r$ is called the {\it floor} of $r$ and denoted $\left \lfloor{r}\right \rfloor $.\\

We are now ready to determine when $L_6(p,q)\in \mathbf{V}_{m,n}$ and, as can be seen, the answer is rather curious.

\begin{Theorem} \label{E3}
Let $3\leq n, q <\omega $, $0\leq m\leq n$ and $0\leq p\leq q$. Then the following are equivalent:\\
{\rm (1)} $L_6(p,q)\in \mathbf{V}_{m,n}$, \\ 
{\rm (2)} $p\le m$, and\\
 either\\
\indent {\rm (a)}  $p = q= m = n$,\\
\noindent or\\
\indent {\rm (b)} $p < q$ and $ q\leq p+ \left \lfloor{(m-p)/2}\right \rfloor + n-m$.

\end{Theorem}
\begin{proof}
Say $Q_6(m,n) = Q_6(I,S )$ where  $I \subseteq S$, $|I| = m$ and $|S| = n$, and
$Q_6(p,q) = Q_6(J,T )$ where $J\subseteq T$, $|J| = p$ and $|T| = q$. We know that $L_6(p,q)\in \mathbf{V}_{m,n}$ if and only if there is a surjective $pm$-morphism from $Q_6(I,S)$ onto $Q_6(J,T)$.
Suppose that (1) holds and let $\varphi$ be such a surjective $pm$-morphism. Then, by (1) of Lemma \ref{E3A}, $\varphi(S)=T$ and $\varphi(\zeta(S))=\zeta(T)$. It is obvious that   $q\leq n$.  Moreover, $\varphi(\varphi^{-1}(J))=J$ and, applying (3) of Lemma \ref{E3A}, we have $p=|J|=|\varphi^{-1}(J)|$. But  $\varphi^{-1}(J)\subseteq I$, by (2) of the same lemma, and  then $p=|\varphi^{-1}(J)|\le |I|= m$. So it is always the case that $p\leq m$. Now we have 2 cases to consider.\\

Case 1:  $p = q$. Then $J=T$ and so, $\varphi^{-1}(J)=\varphi^{-1}(T)=S$. Thus $p=n$ and, since $p\leq m\leq n$, we conclude that $p=m=n$. So $p=q=m=n$, as required.\\

Case 2: $p < q$. To show that $q\leq p+ \left \lfloor{(m-p)/2}\right \rfloor + n-m$, we may consider that $q$ is the greatest integer for which there exists a surjective $pm$-morphism from $Q_6(m,n)$ onto $Q_6(p,q)$. As $\varphi(S)=T$, we have $\varphi(S\setminus \varphi^{-1}(J) )= T\setminus J$. Since $p=|\varphi^{-1}(J)|$ and $\varphi^{-1}(J)\subseteq I$, we have $|I\setminus \varphi^{-1}(J)|=m-p$.  \\

 If $m-p=0$, then $ \left \lfloor{(m-p)/2}\right \rfloor=0$ and
$p+ \left \lfloor{(m-p)/2} \right \rfloor + n-m=p+0+n-p=n$ and we have already observed  that $q\leq n$. \\

Suppose $m-p=1$, that is $|I\setminus \varphi^{-1}(J)|=1$. Then $I\setminus \varphi^{-1}(J)$ is a singleton, say $I\setminus \varphi^{-1}(J)=\{x\}$. By (4) of Lemma \ref{E3A}, there exists $u\in S\setminus \varphi^{-1}(J)$ such that $u\neq x$ and $\varphi(u)=\varphi(x)$. Thus, $q=|J|+|T\setminus J|=p+|\varphi(S\setminus \varphi^{-1}(J))|\leq p+ (|S\setminus \varphi^{-1}(J)|-1)=p+n-p-1=n-1=p+ \left \lfloor{(m-p)/2}\right \rfloor + n-m$.\\

 Finally, suppose that $m-p\geq 2$, that is $|I\setminus \varphi^{-1}(J)|\geq 2$. As  $\varphi^{-1}(J)\subseteq I$, we have $T\setminus J=\varphi(S\setminus \varphi^{-1}(J))=\varphi((I\setminus \varphi^{-1}(J))\cup (S\setminus I))$. Notice  that $|S\setminus I|=n-m$. Consider $\theta$ the equivalence relation on $I\setminus \varphi^{-1}(J)$ defined by $(x,y)\in\theta  \mbox{ iff } \varphi(x)=\varphi(y)$. Suppose that there are $k$ equivalence classes. Then $|T\setminus J|\leq k+n-m$. If each class has at least $2$ elements then $m-p= |I\setminus \varphi^{-1}(J)|\geq 2k$, implying $(m-p)/2\geq k$.  So  $\left \lfloor{(m-p)/2}\right \rfloor \geq k$. Thus, $q=|J|+|T\setminus J|\leq p+k+n-m \leq p+\left \lfloor{(m-p)/2}\right \rfloor+n-m$, as required. Now suppose that there exists a class with cardinality $1$. Then, as $|I\setminus \varphi^{-1}(J)|=m-p\geq 2$, we have that $k\geq 2$. Let $a_1/\theta, a_2/\theta, \ldots, a_k/\theta$ be the equivalence classes with $a_1/\theta=\{a_1\}$. For each $i$, $2\leq i\leq k$, we have $a_1, a_i\in I\setminus\varphi^{-1}(J)$ and $\varphi(a_1)\neq\varphi(a_i)$. By (4) of Lemma \ref{E3A}, there exists $u\in S\setminus \varphi^{-1}(J)$ such that $u\neq a_1$ and $\varphi(u)=\varphi(a_1)$ and there exists $v_i\in S\setminus \varphi^{-1}(J)$ such that $v_i\neq a_i$ and $\varphi(v_i)=\varphi(a_i)$. It is obvious that the elements $u, a_1, a_i, v_i$ are distinct and that $u\notin I$, as $a_1/\theta=\{a_1\}$.  By the maximality of $q$ and applying Lemma \ref{E3B} we conclude that $v_i\in I\setminus \varphi^{-1}(J)$ and consequently, $v_i\in a_i/\theta$. So, all the classes except $a_1/\theta$ have at least $2$ elements and $m-p=|I\setminus \varphi^{-1}(J)|\geq 2(k-1)+1=2k-1$ and $(m-p)/2\geq k-1/2$, so  $\left \lfloor{(m-p)/2}\right \rfloor \geq k-1$. Taking into account that $u\in S\setminus I$ and $\varphi(u)=\varphi(a_1)$, we have $q-p=|T\setminus J|=|\varphi((I\setminus \varphi^{-1}(J))\cup (S\setminus I))|\leq k+(n-m-1)=(k-1)+(n-m)\leq \left \lfloor{(m-p)/2}\right \rfloor+n-m$ and hence (2) follows.\\
 
Now, suppose (2) is true.  We wish to prove that there is a surjective $pm$-morphism from $Q_6(I,S)$ onto $Q_6(J,T)$.  If $p=q=m=n$, it is immediate. Now suppose that $p\leq m$ and $p< q\leq p+ \left \lfloor{(m-p)/2}\right \rfloor + n-m$. 
As for any $t\geq q$, there is a surjective $pm$-morphism from $Q_6(p,t)$ onto $Q_6(p,q)$, it is sufficient to prove the statement for $q= p+ \left \lfloor{(m-p)/2}\right \rfloor + n-m$.\\  

As $p\leq m=|I|$ consider $M\subseteq I$ such that $|M|=p=|J|$ and let $\alpha$ be a bijection from $M$ onto $J$.  Then $|I\setminus M|=m-p=2k$ or $|I\setminus M|=m-p=2k+1$, for some $k\geq 0$, and, in both cases, 
 $ \left \lfloor{(m-p)/2}\right \rfloor =k$.\\
  
 Suppose $k\geq 1$. Let $I\setminus M=\{a_1, \ldots, a_{k}\}\cup\{b_1, \ldots, b_k\}\cup B$, where $B=\emptyset$ if $|I\setminus M|=2k$, and $B=\{b\}$ if $|I\setminus M|=2k+1$. As $|T|=q= p+ k + n-m$ and $|J|=p$, $|T\setminus J|=k+n-m\geq k$. Choose $k$ (distinct) elements in $T\setminus J$, say $c_1, \ldots, c_k$, and let $C=\{c_1, \ldots, c_k\}$.  Let $\beta \colon  I\setminus M \to C$ be defined by $\beta(a_i)=\beta(b_i)=c_i$, for any $1\leq i\leq k$, and $\beta(b)=c_k$ if $B\neq \emptyset$.  As the sets $T\setminus (J\cup C)$ and $S\setminus I$ both have  cardinality $n-m$, consider a bijection $\gamma$ from $S\setminus I$ onto $T\setminus (J\cup C)$. Now let $\varphi\colon Q_6(I,S)\to Q_6(J,T)$ be defined by, for $x\in S$, 
\begin{center}
$\varphi(x)=\begin{cases} \alpha(x) \text{ if } x\in M\\ \beta(x) \text{ if } x\in I\setminus M \\
\gamma(x) \text{ if } x\in S\setminus I,  \end{cases}$\\

\hspace{-7cm}and \\
\hspace{-2cm}$\varphi(\zeta(x))=\zeta (\varphi(x))$.
\end{center}

By Lemma \ref{E3A}, $\varphi$ is a surjective $pm$-morphism.  \\

Now suppose $k=0$,  that is $m-p=0$ or $m-p =1$ and, consequently, $I=M$ or $I\setminus M$ is a singleton, say $I\setminus M=\{b\}$. As $\left \lfloor{(m-p)/2}\right \rfloor =0$, we have that $p+1\leq q=p+n-m$. Then $n-m\geq 1$ and $|T\setminus J|=q-p=n-m=|S\setminus I|$. Choose $c\in T\setminus J$ and let $\beta$ be a bijection from $ S\setminus I$ onto $T\setminus J$. Consider the map $\varphi\colon Q_6(I,S)\to Q_6(J,T)$ defined by, for $x\in S$, 
\begin{center}
$\varphi(x)=\begin{cases} \alpha(x) \text{ if } x\in M\\ 
c \;\;\;\;\;\,\text{ if } x\in I\setminus M\\ \beta(x) \text{ if } x\in S\setminus I, \end{cases}$\\

\hspace{-7cm}and \\
\hspace{-2cm}$\varphi(\zeta(x))=\zeta (\varphi(x))$. 
\end{center}

By Lemma \ref{E3A}, $\varphi$ is a surjective $pm$-morphism.
\end{proof}

\begin{Corollary} \label{E4} We have the following\\
{\rm (1)} For $3\leq p<\omega$, $3\leq n<\omega$ and $0\leq m\leq n$, $L_6(p,p)\in \mathbf{IS}(L_6(m,n))$ if and only if $p=m=n$.\\
{\rm (2)} $|L_V({\mathbf{M_1}})| = 2^\omega$.
\end{Corollary}
\begin{proof} (1) Immediate from the previous theorem.\\

(2)   Suppose $\mathcal{K}_1$ and $\mathcal{K}_2$ are distinct non-empty subsets of $\{L_6(n,n)\colon 3\leq n < \omega\}$.
 Without loss of generality, we may assume that there exists, say, $L_6(n,n)\in \mathcal {K}_1\setminus \mathcal {K}_2$ for some $3\leq n<\omega$.  If $L_6(n,n)\in {\mathbf V}(\mathcal {K}_2)$, then, by Theorem \ref{E0}, $L_6(n,n)\in \mathbf{IS}(\mathcal{K}_2)$, that is $L_6(n,n)\in \mathbf{IS} (L_6(m,m))$ for some $L_6(m,m)\in \mathcal {K}_2$, which is a contradiction, since $n\neq m$. So $L_6(n,n)\notin \mathbf{V}(\mathcal{K}_2)$ and $\mathbf{V}(\mathcal{K}_1)\neq \mathbf{V}(\mathcal{K}_2)$.
\end{proof}

Although $L_V({\mathbf{M_1}})$ is uncountable, with the aid of Theorem \ref{E3}, its structure may still be understood.\\

Observe that ${\mathbf V}_0\subseteq {\mathbf V}$ for any non-trivial ${\mathbf V}\in L_V({\mathbf{M_1}})$.\\

Further, since both $Q_1$ and $Q_3$ have an order component $Q$ for which $Q\cap \zeta (Q) = \emptyset$, we have that $L_1,L_3\not \in {\mathbf V}({\mathcal L})$ for ${\mathcal L}\subseteq Si_F\setminus \{L_1, L_3\}$.\\

First we consider the lattice of subvarieties of $\mathbf{L}=\mathbf{V}(L_0, L_1, L_2, L_3, L_4, L_5)$ which is a finite distributive lattice, whose non-zero join-irreducible elements are $\mathbf{V}_i$, $0\leq i \leq 5$. Inspection shows  that the poset of these elements is\\

\vspace*{0.2in}

\unitlength 1mm 
\linethickness{0.4pt}
\ifx\plotpoint\undefined\newsavebox{\plotpoint}\fi 
\begin{picture}(200.581,100.581)(-15,0)
\put(38,97.75){\circle*{2.5}}
\put(56.75,98.5){\circle*{2.5}}
\put(74,99){\circle*{2.5}}
\put(56.5,82.5){\circle*{2.5}}
\put(39,67.75){\circle*{2.5}}
\put(22.75,82.75){\circle*{2.5}}
\multiput(74,99)(-.03768577495,-.03370488323){942}{\line(-1,0){.03768577495}}
\multiput(38,97.75)(.0403761062,-.0337389381){452}{\line(1,0){.0403761062}}
\multiput(56.25,98.75)(.03125,-2.03125){8}{\line(0,-1){2.03125}}
\multiput(37.5,97.5)(-.0337078652,-.0342696629){445}{\line(0,-1){.0342696629}}
\multiput(22.75,83.25)(.0347826087,-.0336956522){460}{\line(1,0){.0347826087}}
\put(45.75,67.75){\makebox(0,0)[cc]{$\mathbf{V}_0$}}
\put(62,82){\makebox(0,0)[cc]{$\mathbf{V}_2$}}
\put(16.75,81.75){\makebox(0,0)[cc]{$\mathbf{V}_1$}}
\put(31.5,98){\makebox(0,0)[cc]{$\mathbf{V}_3$}}
\put(50.25,98.25){\makebox(0,0)[cc]{$\mathbf{V}_4$}}
\put(68.75,98.5){\makebox(0,0)[cc]{$\mathbf{V}_5$}}
\put(47,58){\makebox(0,0)[cc]{Figure $2$}}
\end{picture}

\vspace*{-2,0in}

The lattice of subvarieties of $\mathbf{L}$ is isomorphic to the lattice of its decreasing sets. We conclude, in particular, that the variety $\mathbf{L}=\mathbf{V}(L_0, L_1, L_2, L_3, L_4, L_5)$ has $14$ non-trivial subvarieties, namely:\\
$\mathbf{V}_i$, $0\leq i \leq 5$,
 $\mathbf{V}_1\vee \mathbf{V}_2$, $\mathbf{V}_1\vee \mathbf{V}_4$, $\mathbf{V}_1\vee \mathbf{V}_5$, $\mathbf{V}_3\vee \mathbf{V}_4$, $\mathbf{V}_3\vee \mathbf{V}_5$, $\mathbf{V}_4\vee \mathbf{V}_5$, $\mathbf{V}_1\vee \mathbf{V}_4\vee \mathbf{V}_5$, $\mathbf{V}_3\vee \mathbf{V}_4\vee \mathbf{V}_5=\mathbf{L}$.\\

For the sake of conciseness in the analysis that follows including Theorem \ref{E6}, and unless otherwise stated,  whenever we consider $L_6(m,n)$ or $\mathbf{V}_{m,n}$ it is implicit that $3\leq n<\omega$ and $0\leq m \leq n$.\\

 Now observe that\par
(1) $\mathbf{V}_4\subseteq \mathbf{V}_{m,n}$ if and only if $m\geq 1$,\par
(2) $\mathbf{V}_5\subseteq \mathbf{V}_{m,n}$ if and only if  $(m,n)\neq (3,3)$. \\

 Let $\mathbf{V}$ be a subvariety of $\mathbf{M}_1$ such that $L_1\notin \mathbf{V}$ and $L_6(m,n)\in \mathbf{V}$, for some $m,n$. Then $L_3\notin \mathbf{V}$.

 Consider 
\begin{center}
$S=\{n\colon L_6(n,n)\in \mathbf{V}\}$ and

\hspace{0.3cm}

$T=\{m\colon  L_6(m,n)\in \mathbf{V},\mbox{ for some } n>m\}$.
\end{center}

We have that $S\neq\emptyset$ or $T\neq\emptyset$.  We analyse  the various possibilities for $S$ and $T$.\\

{\bf (1)} $S=\emptyset$. Then $T\neq \emptyset$.\par
{\bf (1.1)} $T$ is infinite. Consider $L_6(p,q)$ with $p< q$. As $T$ is infinite, there exists $m\in T$ such that $m\geq 2q-p>p$. Since $m\in T$, we have that $L_6(m,n)\in \mathbf{V}$ for some $n>m$ and   $L_6(p,q)\in \mathbf{V}_{m,n}\subseteq \mathbf{V}$, by Theorem \ref{E3}. So $\mathbf{V}_{\omega, \omega}\subseteq \mathbf{V}$. Since  $L_0, L_2, L_4, L_5\in \mathbf{V}_{\omega,\omega}$ and $S=\emptyset$, we conclude that $\mathbf{V}=\mathbf{V}_{\omega, \omega}$.\par
{\bf (1.2)} $T$ is finite.  Let  $T=\{p_0,\ldots, p_s\}$ with $ s<\omega$ and $p_0>p_1>\cdots > p_s$. For each $i\in\{0,\ldots,s\}$, let $A_i=
\{n\colon  n>p_i \mbox{ and } L_6(p_i,n)\in \mathbf{V}\}$ which is a non-empty set.  \par
 
 {\bf (1.2.1)} $A_i$ is finite, for all $i\in\{0,\ldots,s\}$. Let $q_i={\rm max}\,A_i$. By Theorem \ref{E3}, for any $i\in\{0,\ldots,s\}$ and any $n\geq 3$ such that $p_i<n\leq q_i$, we have $L_6(p_i,n)\in \mathbf{V}_{p_i,q_i}\subseteq\mathbf{V}$.
  If $p_0=0$, then $T=\{0\}$ and,
  $\mathbf{V}=\mathbf{V}_{0,q_0}$ or $\mathbf{V}=\mathbf{V}_{0,q_0}\vee \mathbf{V}_4$, since $L_0, L_2, L_5\in \mathbf{V}_{0,q_0}$ but $L_4\notin \mathbf{V}_{0,q_0}$.  
 Let $p_0\geq 1$. As $1\leq p_0<q_0$, we have that $L_0, L_2, L_4, L_5\in \mathbf{V}_{p_0,q_0}$ and then  $\mathbf{V}=\mathbf{V}_{p_0,q_0}\vee \cdots\vee \mathbf{V}_{p_s,q_s}$. \par
 {\bf (1.2.2)} There exists $i$, $0\leq i\leq s$, such that $A_i$ is infinite.  Let $M$ be the least such $i$. For any $q>p_M$, $3\leq q<\omega$, there exists $n\in A_M$ such that $n\geq q$. So $L_6(p_M,n)\in\mathbf{V}$ and, by Theorem \ref{E3}, $L_6(p_M,q)\in \mathbf{V}_{p_M,n}\subseteq \mathbf{V}$. Thus, $\mathbf{V}_{p_M, \omega}\subseteq \mathbf{V}$. Moreover, for $0\leq p\leq p_M$ and $q>p$, $3\leq q<\omega$, taking $n$ such that $n>{\rm max}\,\{2, p_M, q+p_M-p\}$, we have that $L_6(p,q)\in \mathbf{V}_{p_M,n}\subseteq \mathbf{V}_{p_M, \omega}\subseteq \mathbf{V}$.  If $p_0=0$, then $T=\{0\}$ and $\mathbf{V}=\mathbf{V}_{0,\omega}$ or $\mathbf{V}=\mathbf{V}_{0,\omega}\vee \mathbf{V}_4$, since  $L_0, L_2, L_5\in \mathbf{V}_{0,\omega}$ and $L_4\notin \mathbf{V}_{0,\omega}$. Suppose $p_0\geq 1$. If $M=0$, then $\mathbf{V}=\mathbf{V}_{p_0,\omega}$, since $L_0, L_2, L_4, L_5\in \mathbf{V}_{p_0,\omega}$. If $M\geq 1$, as $A_0,\ldots, A_{M-1}$ are finite,  arguing as in (1.2.1), we get $\mathbf{V}=\mathbf{V}_{p_0,q_0}\vee \cdots\vee \mathbf{V}_{p_{M-1},q_{M-1}}\vee \mathbf{V}_{p_M,\omega}$. \par
 {\bf (2)} $S\neq \emptyset$.\par
 {\bf (2.1)} $S$ is infinite. Consider $L_6(p,q)$ with $p< q$. As $S$ is infinite, there exists $n\in S$ such that $n\geq 2q-p$. Then $p< n$. As $n\in S$,  $L_6(n,n)\in \mathbf{V}$. Since $p<n$, $p<q$ and  $p+ \left \lfloor{(n-p)/2}\right \rfloor + n-n \geq p+(q-p)= q$, we have that $L_6(p,q)\in \mathbf{V}_{n,n}$, by Theorem \ref{E3}. So $\mathbf{V}_{\omega, \omega}\subseteq \bigvee_{n\in S} \mathbf{V}_{n,n}$. As $L_0, L_2, L_4, L_5\in \mathbf{V}_{\omega,\omega}$, we conclude that $\mathbf{V}=\bigvee_{n\in S} \mathbf{V}_{n,n}$.\par

{\bf (2.2)} $S$ is finite. Let  $S=\{n_0,\ldots, n_k\}$ with $ k<\omega$ and $n_0>n_1>\cdots > n_k$.\par
{\bf (2.2.1)} $T=\emptyset$. Then, for any $m,n$ such that $L_6(m,n)\in \mathbf{V}$, we have $m=n$. If $n_0\geq 4$, then $L_6(2,3)\in \mathbf{V}_{n_0,n_0}\subseteq \mathbf{V}$, a contradiction. So $n_0=3$ and $S=\{3\}$. As $L_0, L_2, L_4\in \mathbf{V}_{3,3}$ but $L_5\notin \mathbf{V}_{3,3}$, we conclude that $\mathbf{V}=\mathbf{V}_{3,3}$ or $\mathbf{V}=\mathbf{V}_{3,3}\vee \mathbf{V}_5$.\par
{\bf (2.2.2)} $T\neq \emptyset$. We are going to apply what was done in (1). Notice that $L_4\in \mathbf{V}_{n,n}$, for any $3\leq n<\omega$. If $T$ is infinite, then $\mathbf{V}=\mathbf{V}_{\omega, \omega}\vee \mathbf{V}_{n_0,n_0}\vee \cdots\vee \mathbf{V}_{n_k,n_k}$.
If  $T=\{p_0,\ldots, p_s\}$ with $ s<\omega$ and $p_0>p_1>\cdots > p_s$, then $\mathbf{V}=\mathbf{V}_{p_0, q_0}\vee \cdots\vee \mathbf{V}_{p_s,q_s}\vee \mathbf{V}_{n_0,n_0}\vee \cdots\vee \mathbf{V}_{n_k,n_k}$, or $\mathbf{V}=\mathbf{V}_{p,\omega}\vee 
 \mathbf{V}_{n_0,n_0}\vee \cdots\vee \mathbf{V}_{n_k,n_k}$ for some $ p<\omega$, or 
 $\mathbf{V}=\mathbf{V}_{p_0, q_0}\vee \cdots\vee \mathbf{V}_{p_t,q_t}\vee  \mathbf{V}_{p,\omega}\vee\mathbf{V}_{n_0,n_0}\vee \cdots\vee \mathbf{V}_{n_k,n_k}$ for some $t <\omega$ and $p_0>\cdots >p_t>p$.\\

Losing some of the detail from above, the structure of $L_V(\mathbf{M}_1)$ may be summarised as follows.\\

Let $(\mathcal{K}_i\colon i<2^{\omega}) $ be the family of infinite subsets of $\{L_6(n, n)\colon 3\leq n<\omega\}$.\\

\begin{Theorem}\label{E6}  Every  non-trivial subvariety $\mathbf{V}$ of $\mathbf{M}_1$ may be expressed as follows\\
{\rm (1)} $\mathbf{V}_i$, $0\leq i \leq 5$,
 $\mathbf{V}_1\vee \mathbf{V}_2$, $\mathbf{V}_1\vee \mathbf{V}_4$, $\mathbf{V}_1\vee \mathbf{V}_5$, $\mathbf{V}_3\vee \mathbf{V}_4$, $\mathbf{V}_3\vee \mathbf{V}_5$, $\mathbf{V}_4\vee \mathbf{V}_5$, $\mathbf{V}_1\vee \mathbf{V}_4\vee \mathbf{V}_5$, $\mathbf{V}_3\vee \mathbf{V}_4\vee \mathbf{V}_5$.\\
 {\rm (2)} $\mathbf{V}_{3,3}\vee \mathbf{V}_5$, $\mathbf{V}_{0, \omega}\vee \mathbf{V}_4$ and  $\mathbf{V}_{0,q}\vee \mathbf{V}_4$ for some $3\leq q<\omega$.\\
{\rm (3)} $\mathbf{V}_{p_0,q_0}\vee\cdots\vee \mathbf{V}_{p_{s-1},q_{s-1}}\vee \mathbf{V}_{n_0,n_0}\vee\cdots\vee \mathbf{V}_{n_{k-1},n_{k-1}}$, for some $s, k<\omega$  such that $s+k>0$ where  $p_i<q_i$ for any $0\leq i<s$,  $p_0>\cdots >p_{s-1}$ and $n_0>\cdots>n_{k-1}$.\\
{\rm (4)} $\mathbf{V}_{p_0,q_0}\vee\cdots\vee \mathbf{V}_{p_{s-1},q_{s-1}}\vee \mathbf{V}_{p,\omega}\vee \mathbf{V}_{n_0,n_0}\vee\cdots\vee \mathbf{V}_{n_{k-1},n_{k-1}}$, for some $s, k<\omega$, $p<\omega$ where $p<p_i<q_i$, for any $0\leq i<s$,  $p_0>\cdots >p_{s-1}>p$ and $n_0>\cdots>n_{k-1}$.\\
{\rm (5)} $\mathbf{V}_{\omega,\omega} \vee  \mathbf{V}_{n_0,n_0}\vee\cdots\vee \mathbf{V}_{n_{k-1},n_{k-1}}$ for some $k<\omega$ and $n_0>\cdots>n_{k-1}$.\\
{\rm (6)} $\mathbf{V}(\mathcal{K}_i)$ for $i<2^{\omega}$.\\
{\rm (7)} $\mathbf{V}\vee \mathbf{V}_1$ and $\mathbf{V}\vee \mathbf{V}_3$ where $\mathbf{V}$ is one of the ones described in {\rm (2)-(6)}.
\end{Theorem} 

Observe that $\mathbf{M_1} = \mathbf{V}_3\vee \mathbf{V}(\{L_6(n,n)\colon 3\leq n<\omega\})$. \\

Also notice that the representation of the subvarieties described in (3), (4) and, consequently, (7) might be redundant. For example, $\mathbf{V}_{7,8}\vee \mathbf{V}_{3,6}=\mathbf{V}_{7,8}$ and, although $\mathbf{V}_{7,8}\vee \mathbf{V}_{2,6}$ is an irredundant representation, we have that $\mathbf{V}_{7,8}\vee \mathbf{V}_{2,6}\vee\mathbf{V}_{9,9}=\mathbf{V}_{2,6}\vee \mathbf{V}_{9,9}$.\\

It is possible to obtain irredundant representations for all of the aforementioned subvarieties, but that requires some technical arguments and it does not add much to the understanding of the lattice $L_V(\mathbf{M}_1)$. For example, in (3) we obtain $\mathbf{V}_{p_0,q_0}\vee\cdots \mathbf{V}_{p_{s-1},q_{s-1}}\vee \mathbf{V}_{p_s,q_s}\vee\cdots\vee \mathbf{V}_{p_{s+t-1},q_{s+t-1}}\vee \mathbf{V}_{n_0,n_0}\vee\cdots\vee \mathbf{V}_{n_{k-1},n_{k-1}}$, for some $s, t, k<\omega$  such that $ s+t+k>0$ where   $p_0>\cdots >p_{s-1}>n_0\geq p_s>\cdots> p_{s+t-1}$, $n_0>\cdots>n_{k-1}$, and, for any $0\leq i<s+t$, $p_i<q_i$, $q_i>p_i+\left\lfloor{(p_{i-1}-p_i)/2}\right\rfloor+q_{i-1}-p_{i-1}$ and $q_s>p_s+\left\lfloor{(n_0-p_s)/2}\right\rfloor$.\\ 

\section{The lattice of subvarieties of $\mathbf{K}_2$ } \label{SF}

In this section we show that $L_V({\mathbf{K_2}})$ is uncountable. Recall that we already know that $ \mathbf{K_2}$ is not locally finite, as seen in Theorem \ref{D6}, and that $L_V({\mathbf{K_1}})$ is an $\omega + 1$ chain, as seen in Theorem \ref{E2}.
Toward this end, the following sequence of regular $pm$-algebras  plays a crucial role.

\begin{Example} \label{E7} 
{\rm For $2\leq n<\omega$, let $A_n=\{a_i\colon 0\leq i<n\}$ and $B_n=\{b_i\colon 0\leq i<n\}$ be disjoint $n$-element sets. Consider $S_n=A_n\cup B_n$ and a disjoint copy $\zeta(S_n)$. Extend $\zeta$ by defining $\zeta(\zeta(x))=x$, for every $x\in S_n$. Let $\leq$ be the partial order induced by, for $0\leq i,j<n$, \[a_i<\zeta (a_j),\, b_i<\zeta(b_j),\]
 \[a_i<\zeta (b_j) \text{ iff } i\neq j, \text{ and } b_i<\zeta (a_j) \text{ iff } i\neq j.\]\par
 It is straightforward to prove that 
  $P_n = ( S_n\cup\zeta (S_n);\leq ,\zeta )$ is a $pm$-space. Let $K_n=E(P_n)$. Observe that, since $a_i<\zeta (a_i) $ and $b_i<\zeta (b_i)$ for every $0\leq i<n$, $K_n$ is a pseudocomplemented Kleene algebra and, as $P_n$ has height $1$, $K_n$ is a regular pseudocomplemented Kleene algebra.}
\end{Example}

\begin{Lemma}\label{E8A} 
{\rm (1)} For $2\leq n<\omega$, the regular pseudocomplemented Kleene algebra $K_n$ is a simple algebra of $\mathbf{K}_2$.\\
{\rm (2)} For $2\leq n, m <\omega$, if $K_n\in \mathbf{IS}(K_m)$, then $n=m$.
\end{Lemma}
\begin{proof}
(1) Applying Corollary \ref{X6}, it is sufficient to prove that, for any $x,y\in P_n$, we have $\ell(x,y)\leq 2$ or $\ell(x,\zeta(y))\leq 2$. Suppose that, for distinct $i$ and $ j$, $0\leq i, j <n$, $x\in\{a_i, b_i\}$ and $y\in \{a_j, b_j\}$. Then $x<\zeta(x)$ and $\zeta(x)> y$. So $\ell(x,y)=2$. For $0\leq i<n$, we have $ \ell(a_i, b_i)=2$, since, choosing $j\neq i$, $a_i< \zeta (b_j)$ and $\zeta(b_j)> b_i$. 
So, for any $x,y\in S_n$, $\ell(x,y)\leq 2$ and, consequently, $\ell(\zeta(x),\zeta(y))\leq 2$. From this fact we also conclude that if either $x\in S_n$ and $y\in\zeta (S_n)$, or $x\in \zeta(S_n)$ and $y\in S_n$, then $\ell(x,\zeta(y))\leq 2$.\\

(2) Suppose $K_n\in \mathbf{IS}(K_m)$. Then there exists a surjective $pm$-morphism $\varphi\colon P_m \to P_n$. It is obvious that $m\geq n$. Suppose $m>n$. For such a $\varphi$, we know that $\varphi(S_m)=\varphi({\rm Min}(P_m))={\rm Min}(P_n)=S_n$ and $\varphi(\zeta(S_m))=\zeta(\varphi(S_m))=\zeta(S_n)$. As $|S_m|=2m$, $|S_n|=2n$ and $2m>2n$, there exist $x, y\in S_m$ such that $x\neq y$ and $\varphi(x)=\varphi(y)$. Since $x, y$ are distinct elements of $S_m$, we have that, for any $z\in S_m$, $z<\zeta(x)$ or $z<\zeta(y)$, and then $\varphi(z)<\varphi(\zeta(x))=\zeta(\varphi(x))$ or $\varphi(z)<\varphi(\zeta(y))=\zeta(\varphi(y))$. But $\varphi(x)=\varphi(y)$, so, for any $z\in S_m$,  $\varphi(z)<\zeta(\varphi(x))$. If $\varphi(x)=a_i$, for some $0\leq i<n$, then $b_i\notin \varphi(S_m)$, since $b_i\not < \zeta(a_i)$. If $\varphi(x)=b_i$, for some $0\leq i<n$, then $a_i\notin \varphi(S_m)$, since $a_i\not < \zeta(b_i)$. Either way, we have reached a contradiction.  Thus, $m=n$.
\end{proof}  
 
Notice that, for any  $2\leq n<\omega$, $K_n\notin {\mathbf{K}_1}$, since $\ell_{\zeta}(a_0, b_0)=2$, as $\ell(a_0, b_0)=2$ and $\ell(a_0,\zeta (b_0))=3$. \\

Consider ${\mathbf V}(\{ K_n\colon 2\leq n<\omega \})$ which is a subvariety of  $\mathbf{K}_2$.  Although, as shown in Theorem \ref{D6}, $\mathbf{K}_2$ is not locally finite, as we will now prove, ${\mathbf V}(\{ K_n\colon 2\leq n<\omega \})$ is, our aim being to establish Theorem \ref{E8}.\\

\begin{Lemma} \label{E9}
The variety ${\mathbf V}(\{ K_n\colon 2\leq n<\omega \})$ is locally finite.
\end{Lemma}

\begin{proof}
As in Theorem \ref{D3}, it is sufficient to show that for each $N<\omega$, there exists $N^\prime <\omega$, depending on $N$, such that any set of $N$ elements in any one of the pseudocomplemented Kleene algebras $K_n$, $2\leq n<\omega$, generates a subalgebra that has at most  $N^\prime $ elements. \\

Let $2\leq n<\omega$ and consider the $pm$-space $P_n$. Observe that, for each $i$, $0\leq i<n$, we have \[(\zeta(a_i)]=\{\zeta(a_i)\}\cup (S_n\setminus \{b_i\}) \text{ and } (\zeta(b_i)]=\{\zeta(b_i)\}\cup (S_n\setminus \{a_i\}).\]

A claim analogous to the one in Theorem \ref{D3} holds in this context too.\\

For any set $C$, let $\mathcal{P}(C)$ denote the Boolean algebra of all subsets of $C$. \\

\noindent {\bf Claim.} Let $\mathcal{F}$ be a Boolean subalgebra of  the Boolean algebra  $\mathcal{P}(S_n)$ such that, for $0\leq i<n$, $\{a_i\}\in \mathcal{F}$ if and only if $\{b_i\}\in \mathcal{F}$, let $\mathcal{F}_I$ be the set of singletons of $\mathcal{F}$ and

 \[K(\mathcal{F}) :=\mathcal{F}\cup \{\zeta(X)\cup S_n\colon X\in \mathcal{F}\}\cup \{(\zeta(x)]\colon \{x\}\in\mathcal{F}_I\}.\]  Then $K(\mathcal{F})$ is a subalgebra of $K_n$.\\

The elements of $K(\mathcal{F})$ are decreasing sets of $P_n$.\\

As in the proof of the claim in Theorem \ref{D3}, $\emptyset\in K(\mathcal{F})$, $S_n\cup \zeta(S_n)\in K(\mathcal{F})$ and  finite unions and intersections  of elements in $\mathcal{F}\cup \{\zeta(X)\cup S_n\colon X\in \mathcal{F}\}$ belong to $K(\mathcal{F})$. Now, let $X\in \mathcal{F}$ and $\{x\}\in \mathcal{F}_I$. Suppose $x=a_i$, for some $0\leq i<n$ (for $x=b_i$ it is analogous). Then $\{a_i\}\in \mathcal{F}$ and, consequently, $\{b_i\}\in\mathcal{F}$. If $b_i\in X$, then $X\cup (\zeta(a_i)]=X\cup \{\zeta(a_i)\}\cup (S_n\setminus \{b_i\})=\zeta(\{a_i\})\cup S_n\in K(\mathcal{F})$, whilst, if $b_i\notin X$,  $X\cup (\zeta(a_i)]=(\zeta(a_i)]$.
 Moreover, $X\cap (\zeta(a_i)]=X\cap (S_n\setminus \{b_i\})\in\mathcal{F}\subseteq K(\mathcal{F})$. 
 Further, we have $(\zeta(X)\cup S_n)\cup (\zeta(a_i)]=\zeta(X\cup \{a_i\})\cup S_n\in K(\mathcal{F})$ and, if $a_i\in X$, $(\zeta(X)\cup S_n)\cap (\zeta(a_i)]=(\zeta(a_i)]$,  otherwise  $(\zeta(X)\cup S_n)\cap (\zeta(a_i)]=S_n\setminus \{b_i\}\in \mathcal{F}\subseteq K(\mathcal{F})$. 
 Finally, let $\{y\}\in \mathcal{F}_I$ such that $y\neq a_i$. Then $\{y\}\in \mathcal{F}$ and $\{a_i, y\}\in \mathcal{F}$ and we have $(\zeta(a_i)]\cup (\zeta(y)]=\zeta(\{a_i, y\})\cup S_n\in K(\mathcal{F})$. If $y=a_j$ with $j\neq i$, then $\{b_j\}\in \mathcal{F}$ and $(\zeta(a_i)]\cap (\zeta(y)]=(\zeta(a_i)]\cap (\zeta(a_j)]=S_n\setminus \{b_i, b_j\}\in\mathcal{F}\subseteq K(\mathcal{F})$. If $y=b_j$, $0\leq j<n$, then $\{a_j\}\in \mathcal{F}$ and $(\zeta(a_i)]\cap (\zeta(y)]=(\zeta(a_i)]\cap (\zeta(b_j)]=S_n\setminus \{b_i, a_j\}\in\mathcal{F}\subseteq K(\mathcal{F})$. Thus $K(\mathcal{F})$ is closed under unions and intersections.\\
 
Now we show that $K(\mathcal{F})$ is closed under pseudocomplementation. Let $X\in \mathcal{F}$. We have $X^{\ast}=(S_n\cup\zeta(S_n))\setminus [X)=S_n\setminus X$, unless $X=\emptyset$ or $|X|=1$. In the first case, $X^{\ast}=S_n\cup\zeta(S_n)$. In the latter, we have $X=\{a_i\}$ or $X=\{b_i\}$, for some $0\leq i<n$, and then both $\{a_i\}$ and $\{b_i\}$ belong to $\mathcal{F}_I$  and $\{a_i\}^{\ast}=(\zeta(b_i)]\in K(\mathcal{F})$, $\{b_i\}^{\ast}=(\zeta(a_i)]\in K(\mathcal{F})$. For $\zeta(X)\cup S_n$, we have $(\zeta(X)\cup S_n)^{\ast}=\emptyset$. Finally, if $\{a_i\}\in \mathcal{F}_I$, for some $0\leq i<n$, then $\{a_i\}\in \mathcal{F}$ and $\{b_i\}\in \mathcal{F}$, and we have $(\zeta(a_i)]^{\ast}=\{b_i\}\in\mathcal{F}\subseteq K(\mathcal{F})$. Analogously, $(\zeta(b_i)]^{\ast}=\{a_i\}\in\mathcal{F}\subseteq K(\mathcal{F})$, if $\{b_i\}\in \mathcal{F}_I$, for some $0\leq i<n$.\\

It remains to show that $K(\mathcal{F})$ is closed for the de Morgan unary operation. Let $X\in \mathcal{F}$. We have $X^{\prime}=(S_n\cup \zeta(S_n))\setminus \zeta(X)=\zeta(S_n\setminus X)\cup S_n\in K(\mathcal{F})$ and $(\zeta(X)\cup S_n)^{\prime}=(S_n\cup \zeta(S_n))\setminus \zeta(\zeta(X)\cup S_n)=S_n\setminus X\in \mathcal{F}\subseteq K(\mathcal{F})$. For the third type of elements in $K(\mathcal{F})$, we have $\{a_i\}\in \mathcal{F}_I$ if and only if $\{b_i\}\in \mathcal{F}_I$,  and $(\zeta(a_i)]^{\prime}=(\zeta(b_i)]\in K(\mathcal{F})$ and $(\zeta(b_i)]^{\prime}=(\zeta(a_i)]\in K(\mathcal{F})$.\\

The proof of the claim is now complete. \\

Let $N$ be a positive integer. Let $\{ X_i\colon 0\leq i<N\}$ be a set of  decreasing sets of $P_n$. For each $i$, $0\leq i<N$, we have that $X_i=(X_i\cap S_n)\cup (X_i\cap \zeta(S_n))=(X_i\cap S_n)\cup \zeta(\zeta(X_i)\cap S_n)$ and $X_i\cap S_n=(X_i\cap A_n)\cup(X_i\cap B_n)$ and $\zeta(X_i)\cap S_n=(\zeta(X_i)\cap A_n)\cup (\zeta(X_i)\cap B_n)$.
 Consider $T=\{X_i\colon 0\leq i<N\}\cup \{\zeta (X_i) \colon 0\leq i<N\}$. This set has, at most, $N_1 = 2N$ elements.\\
 
 For each $Y\subseteq A_n$ and each $Z\subseteq B_n$, let $B(Y)=\{b_i\colon  a_i\in Y \}$ and $A(Z)=\{a_i\colon b_i\in Z \}$. Clearly, for any $Y\subseteq A_n$ and $Z\subseteq B_n$ we have $A(B(Y))=Y$, $B(A(Z))=Z$, $B(A_n\setminus Y)=B_n\setminus B(Y)$ and $A(B_n\setminus Z)=A_n\setminus A(Z)$. It is also clear that $B(Y_1\cap Y_2)=B(Y_1)\cap B(Y_2)$, $A(Z_1\cap Z_2)=A(Z_1)\cap A(Z_2)$, for any $Y_1, Y_2\subseteq A_n$ and $Z_1, Z_2\subseteq B_n$.  
Consider \[\mathcal{A}=\{X\cap A_n\colon X\in T\}\cup \{A(X\cap B_n)\colon X\in T\} \mbox{ and } \mathcal{B}=\{X\cap B_n\colon X\in T\}\cup \{B(X\cap A_n)\colon X\in T\}.\] 
Obviously, each of the sets $\mathcal{A}$ and $\mathcal{B}$ has, at most, $N_2=2N_1$ elements.  Let $\mathcal{F}_{\mathcal{A}}$ be the Boolean subalgebra of the Boolean algebra $\mathcal{P}(A_n)$ generated by $\mathcal{A}$ and $\mathcal{F}_{\mathcal{B}}$ be the Boolean subalgebra of the Boolean algebra $\mathcal{P}(B_n)$ generated by $\mathcal{B}$ and let $\mathcal{F}=\{Y\cup Z\colon Y\in \mathcal{F}_{\mathcal{A}} \mbox{ and } Z\in \mathcal{F}_{\mathcal{B}}\}$ which is a Boolean subalgebra of $\mathcal{P}(S_n)$, since $S_n$ is the disjoint union of $A_n$ and $B_n$. Since a free Boolean algebra on $N_2$ generators has $2^{(2^{N_2})}$ elements, each one of $\mathcal{F}_{\mathcal{A}}$ and $\mathcal{F}_{\mathcal{B}}$ has, at most, $N_3=2^{(2^{N_2})}$ elements, and so, the cardinality of $\mathcal{F}$ is less than or equal to $N_4=(N_3)^2$. Now, suppose that $\{a_i\}\in \mathcal{F}$, $0\leq i<n$. Then $\{a_i\}\in \mathcal{F}_{\mathcal{A}}$ and, since  $\mathcal{F}_{\mathcal{A}}$ is generated by $\mathcal{A}$ and
$\{a_i\}$ is an atom, we have $\{a_i\}=W_1\cap\cdots\cap W_k$, where $1\leq k<\omega$, and, for each $j$, $1\leq j\leq k$, $W_j\in  \mathcal{A}$ or $A_n\setminus W_j\in \mathcal{A}$. So $\{b_i\}=B(\{a_i\})=B(W_1)\cap\cdots\cap B(W_k)$ and, for each $j$, $1\leq j\leq k$, $B(W_j)\in  \mathcal{B}$ or $B_n\setminus B(W_j)= B(A_n\setminus W_j)\in \mathcal{B}$. Thus $\{b_i\}\in\mathcal{F}_{\mathcal{B}}\subseteq \mathcal{F}$. Analogously, we prove that if $\{b_i\}\in \mathcal{F}$, $0\leq i<n$, then $\{a_i\}\in \mathcal{F}$.\\

By the above claim, $K(\mathcal{F})$ is a subalgebra of $K_n$. This subalgebra has, at most, $3N_4$ elements.
It remains to show that, for any $i$, $0\leq i<N$, $X_i\in K(\mathcal{F})$. We have that $X_i=(X_i\cap S_n)\cup \zeta(\zeta(X_i)\cap S_n)$ and $X_i\cap S_n=(X_i\cap A_n)\cup(X_i\cap B_n)$ and $\zeta(X_i)\cap S_n=(\zeta(X_i)\cap A_n)\cup (\zeta(X_i)\cap B_n)$.
 If $\zeta(X_i)\cap S_n=\emptyset$, then  $X_i=(X_i\cap A_n)\cup(X_i\cap B_n)\in \mathcal{F}\subseteq K(\mathcal{F})$ since, as $X_i\in T$,  $X_i\cap A_n\in \mathcal{A}\subseteq \mathcal{F}_{\mathcal{A}}$ and $X_i\cap B_n\in \mathcal{B}\subseteq \mathcal{F}_{\mathcal{B}}$. Now, suppose $\zeta(X_i)\cap S_n=\{a_j\}$ for some $0\leq j<n$. Then $\zeta(X_i)\cap A_n=\{a_j\}$. As $\zeta(X_i)\in T$, we have $\{a_j\}\in \mathcal{A}\subseteq \mathcal{F}_{\mathcal{A}}\subseteq \mathcal{F}$. As $\{a_j\}$ is a singleton, it belongs to $\mathcal{F}_I$. Since $\zeta(X_i)\cap S_n=\{a_j\}$, we have $\zeta(\zeta(X_i)\cap S_n)=\{\zeta(a_j)\}$ and $X_i=(X_i\cap S_n)\cup\{\zeta(a_j)\}$. But $X_i$ is decreasing, so  either $X_i=\{\zeta(a_j)\}\cup (S_n\setminus \{b_j\})=(\zeta(a_j)]$ or $X_i=\zeta(\{a_j\})\cup S_n$. Either way, $X_i\in K(\mathcal{F})$. We prove similarly that if $\zeta(X_i)\cap S_n=\{b_j\}$ for some $0\leq j<n$, then $X_i\in K(\mathcal{F})$. Finally, suppose $|\zeta(X_i)\cap S_n|\geq 2$. As $X_i\cap \zeta(S_n)=\zeta(\zeta(X_i)\cap S_n)$, we have $|X_i\cap \zeta(S_n)|\geq 2$ and, since $X_i$ is decreasing, $S_n\subseteq X_i$. So $X_i=\zeta(\zeta(X_i)\cap S_n)\cup S_n$. But $\zeta(X_i)\cap S_n=(\zeta(X_i)\cap A_n)\cup (\zeta(X_i)\cap B_n)$ and, as $\zeta(X_i)\in T$, we have $\zeta(X_i)\cap A_n\in \mathcal{A}\subseteq \mathcal{F}_{\mathcal{A}}$ and $\zeta(X_i)\cap B_n\in \mathcal{B}\subseteq \mathcal{F}_{\mathcal{B}}$. Consequently, $\zeta(X_i)\cap S_n \in \mathcal{F}$ and $X_i\in K(\mathcal{F})$.
 \end{proof}

\begin{Lemma} \label{E10}
$|L_V({\mathbf V}(\{ K_n\colon 2\leq n<\omega \}))| = 2^\omega$.
\end{Lemma}
\begin{proof}  Let $\mathcal{K}$ and $\mathcal{K}'$ be non-empty subsets of $\{ K_n\colon 2\leq n<\omega \}$ such that $\mathcal{K}\neq \mathcal{K}'$.
Without loss of generality, we may assume that there exists $K_n\in \mathcal{K}\setminus \mathcal{K}'$. Obviously, $K_n\in \mathbf{V}(\mathcal{K})$. If $K_n\in \mathbf{V}(\mathcal{K}')$, then, applying Theorem \ref{E0}, $K_n\in \mathbf{IS}(K_m)$ for some $K_m\in \mathcal{K}'$ and, by (2) of Lemma \ref{E8A}, $n=m$ which is a contradiction.  Thus, $K_n\notin \mathbf{V}(\mathcal{K}')$ and $\mathbf{V}(\mathcal{K})\neq  \mathbf{V}(\mathcal{K}')$.
\end{proof}

Since ${\mathbf V}(\{ K_n\colon 2\leq n<\omega \})$ is a subvariety of $\mathbf{K}_2$, the following is immediate.\\

\begin{Theorem} \label{E8}
$|L_V(\mathbf{K}_2)| = 2^\omega$.\hfill{$\Box$}
\end{Theorem}

The  ad hoc nature of Example \ref{E7} strongly suggests that, unlike $L_V({\mathbf{M_1}})$ which is also uncountable, the structure of $L_V({\mathbf{K_2}})$ (and so too $L_V({\mathbf{M_2}})$) defies a meaningful description.\\

\small

\ \\
\ \ \ \ \ \ \ \ \\
\begin{tabbing}
    Department of Mathematics \hspace*{18em}\=Centro de Matem\a'{a}tica e Aplica\c{c}\~{o}es \\
   State University of New York  \>  Departamento de Matem\a'{a}tica\\
   New Paltz, NY 12561, USA \>  Faculdade de Ci\^{e}ncias e Tecnologia\\
   \ \> Universidade NOVA de Lisboa\\
   adamsm@newpaltz.edu  \> 2829-516 Caparica, Portugal\\
   
    sankapph@newpaltz.edu \> \ \\
    \ \> jvc@fct.unl.pt\\

\end{tabbing}

\end{document}